\documentclass[12pt, reqno]{amsart}

\usepackage{lscape,amssymb, amsmath,verbatim,graphicx}
\usepackage{epsfig,subfigure,float,subfigure}
\usepackage{graphicx}
\usepackage{epsfig}
\newcommand{\lf}{\lfloor}
\newcommand{\rf}{\rfloor}

\numberwithin{figure}{section}
\newtheorem{theorem}{Theorem}[section]
\newtheorem{lemma}{Lemma}[section]

\newtheorem{example}{Example}[section]
\newcommand{\vare}{\varepsilon}

\newcommand{\hht}{\hat{t}}

\newcommand{\bark}{\bar{\kappa}}

\newtheorem{assumption}{Assumption}[section]

\newcommand{\calS}{{\mathcal S}}
\newcommand{\calU}{{\mathcal U}}
\newcommand{\cH}{{\mathcal H}}
\newcommand{\calu}{{\frak u}}

\newtheorem{remark}{Remark}[section]
\allowdisplaybreaks
\def\beq{\begin{equation}}
\def\eeq{\end{equation}}
\def\bals{\begin{align*}}
\def\eals{\end{align*}}

\def\bal{\begin{align}}
\def\eal{\end{align}}

\numberwithin{equation}{section}
\numberwithin{theorem}{section}
\numberwithin{table}{section}
\numberwithin{remark}{section}

\begin{document}

\title[Estimation of the time of change  in panel data]{Estimation of the time of change  in panel data}

\author {Lajos Horv\'ath}
\address{Lajos Horv\'ath, Department of Mathematics, University of Utah, Salt Lake City, UT 84112--0090 USA }

\author{Marie Hu\v{s}kov\'a}
\address{Marie Hu\v{s}kov\'a, Department of Probability and Mathematical Statistics, Charles University, Sokolovsk\'a 83, CZ-18600 Praha, Czech Republic}
\author {Gregory Rice}
\address{Gregory Rice, Department of Mathematics, University of Utah, Salt Lake City, UT 84112--0090 USA }
\author{ Jia Wang}
\address{Jia Wang, Department of Mathematics, University of Utah, Salt Lake City, UT 84112--0090 USA}


\thanks{ Address correspondence to Lajos Horv\'ath,  Department of Mathematics, University of Utah, Salt Lake City, UT 84112--0090 USA; email: horvath@math.utah.edu.}
\begin{abstract}
We consider the problem of estimating the common time of a change in the mean parameters of panel data when dependence is allowed between the panels in the form of a common factor. A CUSUM type estimator is proposed, and we establish first and second order asymptotics that can be used to derive consistent confidence intervals for the time of change. Our results improve upon existing theory in two primary directions. Firstly, the conditions we impose on the model errors only pertain to the order of their long run moments, and hence our results hold for nearly all stationary time series models of interest, including nonlinear time series like the ARCH and GARCH processes. Secondly, we study how the asymptotic distribution and norming sequences of the estimator depend on the magnitude of the changes in each panel and the common factor loadings. The performance of our results in finite samples is demonstrated with a Monte Carlo simulation study, and we consider applications to two real data sets: the exchange rates of 23 currencies with respect to the US dollar, and the GDP per capita in 113 countries.

\end{abstract}

\maketitle

\section{INTRODUCTION }\label{sec-main}

In this paper, we consider the problem of estimating the time of a change in the mean present in panel data in which their are $N$ panels comprised of time series data of length $T$. A common structural break in panel data is a quite natural occurrence. For example, if each panel represents the exchange rate of a currency with respect to US dollars, then a crisis in the US would be expected to simultaneously affect each panel. Similar phenomena may be produced my governmental policy changes, the introduction of a new technology, etc., and in these cases it is of interest to estimate the time at which such occurrences are manifested in sample data. The theory of change point analysis has been extensively developed to study problems of this nature, see Cs\"{o}rg\H{o} and Horv\'ath (1997), Brodsky and Darkhovskii (2002), and Aue and Horv\'ath (2012) for reviews of the field.

Classical methods in change point analysis consider univariate and multivariate data of a fixed dimension. In many panel data examples, however, the number of panels $N$ is comparable in size to the length of the series $T$. In these cases asymptotics as $T$ remains fixed and $N$ tends to infinity, or as $N$ and $T$ jointly tend to infinity, are more appropriate. Although in principle one could detect the common change present in each panel by examining a single series, an analysis that utilizes all available panels should provide improved detection and estimation.

The literature on structural breaks in panel data has grown considerably in the last two decades. We refer to Arellano (2003), Hsiao (2003) and Baltagi (2013) for surveys of several panel data models and their applications to econometrics and finance. The early foundations for estimating structural changes in panel data were developed in Joseph and Wolfson (1992,1993), and many aspects of the problem have now seen at least some consideration; Li et al. (2014) and Qian and Su (2014) consider multiple structural breaks in panel data, and Kao et al. (2014) considers break testing under cointegration.

The works of Bai (2010), Kim (2011, 2014), and Horv\'ath and Hu\v{s}kov\'a (2012) are the most closely related to the present paper. Bai (2010) considers the problem of estimating a common break in the means of panel data that do not exhibit cross sectional dependence. A least squares estimator is proposed that is shown to be consistent when $N$ tends to infinity, and its asymptotic properties are derived as $N$ and $T$ jointly tend to infinity. Kim (2011) considers the least squares estimator of Bai (2010) with cross sectional dependence modeled by a common factor, and expands the test to detect a change in the slope of a linear trend in the mean component. Horv\'ath and Hu\v{s}kov\'a (2012) study testing for the presence of a change using a CUSUM estimator, and also assuming the presence of a common factor. In each of these papers, asymptotics are derived assuming the model errors are linear processes, and that the rates of divergence relative to $N$ and $T$ of the size of the changes and the magnitudes of the factor loadings are fixed.

In this paper, we expand the existing theory in two primary directions. We derive second order asymptotics for the CUSUM change point estimator assuming only an order condition on the long run moments. This extends the asymptotic theory of change point estimation to a wide variety of error processes, notably many nonlinear time series examples like the ARCH and GARCH processes. We also show explicitly how the asymptotic distribution and norming sequences of the estimator depend on the magnitude of the changes in each panel and the common factor loadings. This allows for the computation of the limit distribution under several conceivable rates of divergence for the magnitudes of changes and factor loadings.

The remainder of the paper is organized as follows. In Section \ref{sec-main-res}, we present our assumptions and the main results of the paper. Section \ref{sec-exa} contains examples of error processes that satisfy the assumptions of Section \ref{sec-main-res}. Estimators for the norming sequences that appear in the results of Section \ref{sec-main-res} are developed and studied in Section \ref{subest}. The implementation of the results of the paper as well as a Monte Carlo simulation study and data applications are detailed in Section \ref{sec-simu}. The proofs of all results are contained in Appendix \ref{proof}.


\section{Assumptions, and main results }\label{sec-main-res}
We consider the panel data model
\beq\label{model}
X_{i,t}=\mu_i+\delta_iI\{t> t_0\}+\gamma_i\eta_t+e_{i,t},\;\;1\leq i \leq N, 1\leq t \leq T,
\eeq
where the idiosyncratic errors $e_{i,t}'s$ have mean zero, $\eta_t$ denotes the common factor with loadings $\gamma_i$, $1\le i \le N$, and $\delta_i$ denotes the change in the mean of panel $i$ that occurs at the common, and unknown, change point $t_0$.

\begin{assumption}\label{b-1}\ \\
{\rm
(i) The sequences $\{ e_{i,t}, -\infty<t<\infty\}, 1 \leq i \leq N\;\; \mbox{are independent,}$ and\\

(ii) $\{\eta_t, -\infty<t<\infty\}\mbox{ and }\{e_{i,t}, -\infty<t<\infty\}\;\;1\le i \le N \mbox{ are stationary.}$
}
\end{assumption}


According to Assumption \ref{b-1}(i), the only source of dependence between the panels is the common factor $\eta_t$. The idiosyncratic errors form a stationary time series, similarly to the assumption in Bai (2010) and Kao et al.\ (2012, 2013). Throughout this paper $\delta_i$ and $\gamma_i$, $1\le i \le N$, are allowed to depend on $N$ and $T$. For the sake of simplicity, we consider the case when $\gamma_i\in R$, but our results could be extended to the more general case of a vector valued factor loading and common factor.


\begin{assumption}[]\label{t-0}
{\rm The time of change in the mean $t_0$ satisfies
$$
 t_0=\lfloor T\theta\rfloor\;\;\;\mbox{with some}\;\;\;0<\theta<1.
$$
}
\end{assumption}

Assumption \ref{t-0} is standard in change point analysis, and corresponds with the assumptions of Bai (2010), Kim (2011,2014), and Horv\'ath and Hu\v{s}kov\'a (2012). It is of interest in some econometric applications to allow for $\theta$ to depend on $N$ and $T$ and tend to the end points $0$ or $1$ at a certain rate; see Andrews (2003) and, in the panel data setting, Qian and Su (2014). The consideration of this problem in generality for our estimator is not a goal of the present paper, and requires a thorough study.

Our estimator for $t_0$ is defined as the location of the maximum of the sum of the CUSUM processes across the panels:

$$
\hat{t}_{N,T}=\mbox{argmax}_{1\leq t <T}
\sum_{i=1}^N\left(S_i(t)-\frac{t}{T}S_i(T)\right)^2,
$$
where
$$
S_i(t)=\sum_{s=1}^tX_{i,s}.
$$

The estimator of Bai (2010) is
\beq\label{bai-def}
t^*_{N,T}=\mbox{argmax}_{1\leq t<T}\sum_{i=1}^N\left(S_i(t)-\frac{t}{T}S_i(T)\right)^2\frac{1}{(t(T-t))},
\eeq
which is the maximum likelihood estimator for $t_0$ assuming that the panels are independent and normally distributed with the same variance, while $\hat{t}_{N,T}$  maximizes the weighted log likelihood.

We impose only conditions on the long run moments of the error processes for our asymptotic results. The long run moments of the errors in panel $i$ are defined by
$$
U_{i,\nu}(t)=E\left|\sum_{s=1}^te_{i,s}\right|^\nu,\;\;1\leq i \leq N,
$$
and we assume that they satisfy the following conditions:
\begin{assumption}[]\label{cond-1}\ \\
{\rm
(i)$\mbox{ There exists}\;\;\sigma_i, 1\leq i \leq N \;\;\mbox{such that}$
$$\max_{1\leq i \leq N}\sup_{1\leq t \leq T} \left|\frac{1}{t}U_{i,2}(t)-\sigma_i^2\right|=o(1),$$
$\mbox{where}\;\;C_1\leq \sigma_i\leq C_2\;\;\mbox{for all}\;\;1\leq i \leq N\;\;\mbox{with some}\;\;0<C_1\leq C_2<\infty,$ and


(ii)
$$\frac{1}{N}\sup_{1\leq t \leq T}\sum_{i=1}^N\left(\frac{1}{t^{\kappa/2}}U_{i,\kappa}(t)\right)^2=O(1)\mbox{  with some  } \kappa>4.
$$

}
\end{assumption}

Additionally, we must assume an analogous condition on the common factors:
\begin{assumption}\label{con-eta}{\rm $\mbox{With some}\;\bar{\kappa}>2.$
$$E\eta_t=0,\; E\left(\sum_{s=1}^t\eta_s\right)^2=t+o(t),\;\;\mbox{and}\;\;E\left|\sum_{s=1}^t\eta_s\right|^{\bar{\kappa}}=O(t^{\bar{\kappa}/2}),\;\mbox{as}\;t\to\infty.$$

}
\end{assumption}

Assumptions \ref{cond-1} and \ref{con-eta} do not assume any specific structure on the error terms, in contrast to the structural break literature with panel data to date. We provide several examples in Section \ref{sec-exa}, including linear and nonlinear time series, martingales and mixing sequences,  where Assumptions \ref{cond-1} and \ref{con-eta} are satisfied.\\


The size of the changes and the correlation between the panels will play a crucial role in the asymptotic distribution of the estimator, and these quantities will be measured by
$$
\Delta_{N,T}=\sum_{i=1}^N\delta_i^2,
\;\;\;
\Gamma_{N,T}=\sum_{i=1}^N\gamma_i^2
\;\;\;
\mbox{and}
\;\;\;
\Sigma_{N,T}=\sum_{i=1}^N\delta_i\gamma_i.
$$
The limit results below are proven when $\min(N,T)\to \infty$.
\begin{assumption}\label{cond-4}
{\rm As $\min(N,T)\to \infty$,


$$ \hspace{-7.25cm}\mbox{(i)} \hspace{6.25cm}\frac{T\Delta_{N,T}}{N}\to\infty,$$

\noindent
and

$$\hspace{-6.80cm}\mbox{(ii)} \hspace{6.25cm}\frac{\Gamma_{N,T}}  {(T\Delta_{N,T})^{1/2}}\to 0.$$
}
\end{assumption}
Assumption \ref{cond-4} means that the sizes of all changes cannot be too small and that the factor loadings cannot be much larger than the sample size and the size of the changes.  Bai (2010) assumes that $\Delta_{N,T}/N$ converges to a positive limit while under the assumptions of Kim (2011), the common factor dominates. A primary goal of our paper is to show how the relationship between the loadings and the sizes of the changes affect the limit distribution of the time of change estimator.

\medskip

Our first result pertains to the asymptotic distribution of $\hat{t}_{N,T}$ when $\Delta_{N,T}$ is large.

\begin{theorem}\label{rate-con-1} If Assumptions \ref{b-1}--\ref{cond-4},
\beq\label{m-cho-2}
\Delta_{N,T}\to \infty
\eeq
and
\beq\label{n-1}
\frac{\Sigma_{N,T}}{\Delta_{N,T}}=o(1)
\eeq
as $N,T\to \infty$, then we have that
\beq\label{new-consi}
P\{\hat{t}_{N,T}=t_0\}\to 1.
\eeq
\end{theorem}


The assumption in \eqref{n-1} may seem somewhat restrictive since it rules out the example of fixed break sizes and factor loadings.  We note that due to the result on page 635 of Horv\'ath and Hu\v{s}kov\'a (2012), when the factor loadings are fixed the CUSUM test for the presence of a change point will reject with probability tending to one regardless of if a change exists or not, and so something along the lines of \eqref{n-1} must be assumed for the CUSUM estimator of the time of change to be consistent.

\begin{remark}\label{rem-bai} {\rm Assume that $T$ is fixed. If Assumptions \ref{b-1}--\ref{con-eta} are satisfied and $\Gamma_{N,T}/\Delta_{N,T}\to 0,$ and  $\Delta_{N,T}/N\to \infty$, then \eqref{new-consi} holds.
}
\end{remark}

\medskip

\begin{remark}\label{rem-bai-cor} {\rm In order to establish the consistency of Bai's (2010) estimator in \eqref{bai-def} for fixed $T$ and under our assumptions we must assume in addition that for each $i$, $e_{i,t}$ and $\eta_t$ are uncorrelated random variables, and that  $\{e_{i,t},  0\leq t<\infty, 1\leq i \leq N\}$ and $\{\eta_t, 0\leq t <\infty\}$ are independent. If in addition to Assumptions \ref{b-1}--\ref{con-eta},  $ E\eta_0^4<\infty,$ $\Delta_{N,T}/N^{1/2}\to\infty$
and $\Gamma_{N,T}/\Delta_{N,T}\to 0$ hold,
then we have that
\beq\label{bai-consi}
\lim_{N\to\infty}P\{t^*_{N,T}=t_0\}=1.
\eeq
We provide a proof of Remark \ref{rem-bai-cor} in Appendix \ref{proof}.
}
\end{remark}

\medskip

The main difference between Remarks \ref{rem-bai} and \ref{rem-bai-cor} is in the assumptions $\Delta_{N,T}/N\to\infty$ and $\Delta_{N,T}/N^{1/2}\to\infty$. Remark \ref{rem-bai-cor} allows smaller changes to establish consistency but much stronger assumptions on the sequences $e_{i,t}, 1\leq i \leq N$ and $\eta_t$. If we cannot assume that $e_{i,t}, 1\leq i \leq N$ and $\eta_t$ are sequences of uncorrelated random variables and the independence of $\{\eta_t, t\geq 0\}$ and $\{e_{i,t}, t\geq 0, 1\leq i \leq N\}$, then \eqref{bai-consi} can be proven under conditions of Remark \ref{rem-bai}. In this case \eqref{new-consi} and \eqref{bai-consi} can hold only if $\Delta_{N,T}/N\to \infty$ when $T$ is fixed.\\

\medskip

We now turn to the asymptotic distribution of $\Delta_{N,T}(\hat{t}_{N,T}-t_0)$ when \eqref{m-cho-2} does not hold, i.e.\ the sizes of the changes are small or occur in only a few panels:

\begin{assumption}\label{m-cho-1}\ \\
{\rm
(i) $\Delta_{N,T}=O(1),$\\

(ii) $T\Delta_{N,T}(1+\log(T/\Delta_{N,T}))^{-2/\bark} \;\;\to \infty,\;
\mbox{where}\;\bark\; \mbox{is defined in Assumption } \ref{con-eta},$ and\\

(iii) ${T^{1-2/\kappa}\Delta_{N,T}}/{N^{1/2}} \;\;\to \infty,\;
\mbox{where}\;\kappa\; \mbox{is defined in Assumption } \ref{cond-1}(ii).$\\

}
\end{assumption}

By Assumption \ref{cond-4}, we have that $T\Delta_{N,T}\to \infty$, so Assumption \ref{m-cho-1}(ii) holds if,  $\Delta_{N,T}T(\log T)^{-2/\bar{\kappa}}\to \infty$. If $N/T^{4/\kappa}\to \infty$, i.e.\ the number of the panels is large, then Assumption \ref{m-cho-1}(iii) follows from Assumption \ref{cond-4}. However, Assumption \ref{cond-4}(ii) also holds if the number of panels is relatively small with respect to the length of the panels and the sizes of the changes.\\
Next we introduce an assumption that is a companion to Assumption \ref{cond-1}:

\begin{assumption}\label{cond-tau}\ \\

{\rm
(i)  $\displaystyle \frac{1}{N}\sum_{i=1}^N|Ee_{i,0}e_{i,t}|=O(t^{-\tau})\;\;\mbox{with some  }\tau>2,$ and,

(ii) $\displaystyle \max_{1\leq i \leq N}U_{i,\bar{\tau}}(t)=O(t^{\bar{\tau}/2})\;\;\mbox{with some}\;\bar{\tau}>2.$
}
\end{assumption}

Assumption \ref{cond-tau} requires an upper bound for the average correlation of the errors in the panels and a uniformity condition that augments Assumption \ref{cond-1}(ii).\\

Our first result in this direction covers the case when the sizes of the changes are small and the effect of the correlation between the panels is negligible or moderate. We measure the dependence between the panels with respect to the sizes of the changes by
$$
{\frak s}=\lim_{\min(N,T)\to \infty}\frac{\Sigma_{N,T}}{\Delta^{1/2}_{N,T}}.
$$
To describe the limit distribution of $\hat{t}_{N,T}$ we need to introduce a drift function
\begin{displaymath}
g_\theta(u)=\left\{
\begin{array}{ll}
(1-\theta)|u|,\;\;\;&\mbox{if}\;\;u< 0
\vspace{0.3cm}\\
\theta u,&\mbox{if}\;\;u\geq 0
\end{array}
\right.
\end{displaymath}
and an asymptotic variance term
$$
\sigma^2=\lim_{N,T\to \infty}\frac{1}{\Delta_{N,T}}\sum_{i=1}^N\delta_i^2\sigma^2_i.
$$
We note that by Assumption \ref{cond-1}(i) we get that $C_1\leq \sigma\leq C_2$.  Let
\begin{displaymath}
\calU_i(t)=\left\{
\begin{array}{ll}
\displaystyle \sum_{s=1}^te_{i,s},\;\;&\mbox{if}\;\;t=1,2,3,\ldots
\vspace{.2 cm}\\
0,&\mbox{if}\;\;t=0
\vspace{.2 cm}\\
\displaystyle -\sum_{s=t}^{-1}e_{i,s},\;\;&\mbox{if}\;\;t=-1,-2,-3,\ldots
\end{array}
\right.
\end{displaymath}
The function ${{\frak {u}}}(s,t)$ is  the asymptotic covariance of $\sum_{i=1}^N\delta_i\calU_i(t)$, i.e.\ for all integers $s$ and $t$
\beq\label{covcal}
\calu(s,t)=\lim_{N\to \infty}E\left(\sum_{j=1}^N \delta_j\calU_j(s)\right)\left(\sum_{i=1}^N \delta_i\calU_i(t)\right).
\eeq
We note that $\calu(s,t)=\lim_{N\to \infty}\sum_{i=1}^N \delta_i^2E\calU_i(s)\calU_i(t)$. It follows from Assumption \ref{cond-1}(i) that the covariance function $\calu(s,t)$ is finite for all integers $s$ and $t$. This function only appears in Theorem \ref{w-fac} below when $\Delta_{N,T}$ is above some positive bound for all $N$ and $T$. In this case $\calu(t,t)>0,$ if $t\neq 0.$\\
The next result considers the case when the common factors are negligible.
\medskip
\begin{theorem}\label{w-fac} We assume that Assumptions \ref{b-1}--\ref{cond-tau} hold,
\beq\label{bar-tau}
 \Delta^{-\bar{\tau}/2}_{N,T}\sum_{i=1}^N|\delta_i|^{\bar{\tau}}\to 0,
 \eeq
 where $\bar{\tau}$ is defined in Assumption \ref{cond-tau}(ii), and
 \beq\label{new-rat}
 {\frak s}=0.
 \eeq
(a) If
\beq\label{w-1}
\Delta_{N,T}\to 0,
\eeq
then we have
\beq\label{f-1}
\frac{\Delta_{N,T}(\hat{t}_{N,T}-t_0)}{\sigma^2}\;\;\;\stackrel{{\mathcal D}}{\to}\;\;\mbox{{\rm argmax}}_u \left\{W(u)-g_\theta (u)   \right\},
\eeq
where $W(u), -\infty<u<\infty$ is a two--sided Wiener process.\\
(b) If
\beq\label{conv2del}
\Delta_{N,T}\to {\frak d}\in (0,\infty),
\eeq
then we have
\beq\label{f-2}
\hat{t}_{N,T}-t_0\;\;\;\stackrel{{\mathcal D}}{\to}\;\;\mbox{{\rm argmax}}_t \{{\frak G}(t)-{\frak d}g_\theta(t)\},
\eeq
where ${\frak G}(t), t=0, \pm 1, \pm 2, \ldots $ is Gaussian with $E{\frak G}(t)=0$ and $E{\frak G}(s){\frak G}(t)=\calu(s,t)$.
\end{theorem}

\begin{remark}{\rm Since the proofs of Theorem \ref{w-fac} and the results to follow depend on normal approximations for the sums $\sum_{i=1}^N \sum_{s=1}^te_{i,s}$  and $\sum_{i=1}^N( \sum_{s=1}^te_{i,s})^2$, the independence of $\{e_{i,t}, t\geq 0\}$ on $i$ could be relaxed as it is pointed out by Bai (2010) and Kim (2011). This would be an important consideration, for example, if the panels are indexed by locations, i.e.\ $i={\bf i}$ describes the location of the panel. In this case a spatial structure could be assumed on the errors. If Assumption \ref{b-1}(i) is replaced by a weak dependence or spatial assumption, the norming constants would change in our limit theorems.
}
\end{remark}

\begin{remark}\label{rem-first} {\rm If $\Delta^{-\bar{\tau}/2}_{N,T}\sum_{i=1}^N|\delta_i|^{\bar{\tau}}\to 0$ with some $\bar{\tau}>2$ does not hold, then the limit in \eqref{f-2} might not be the argmax of a normal process with a drift. For example, if $\delta_1=1$ and $\delta_i=0$ for all $i\geq 2$, then the limit in \eqref{f-2} is determined by the error terms $e_{1,t}, -\infty<t<\infty$.}
\end{remark}
\begin{remark}\label{rem-sec} {\rm The distribution  of $\mbox{{\rm argmax}}_u \{W(u)-g_\theta (u)\}$ is known explicitly. Its density was derived by Ferger (1994) from Corollary 4 of Bhattacharya and Brockwell (1976) (cf.\ Cs\"org\H{o} and Horv\'ath (1997, p.\ 177)).}

\end{remark}
\begin{remark}\label{rem-three} {\rm If \eqref{conv2del} holds, $\gamma_i=0$ for all $1\leq i \leq N$ and $Ee_{i,0}e_{i,t}=0$ for all $1\leq i \leq N$ and $t\neq 0$, then we get the analogue of Theorem 4.2 of Bai (2010). In this case ${\frak G}(s)$ is a Wiener process on integers, so the main difference between the limits in \eqref{f-1} and \eqref{f-2} is that the argmax is computed on the real line or on  integers.}
\end{remark}

So far the common factor was part of the error term and it had a negligible contribution to the limit distribution in Theorem \ref{w-fac}. However,  it has been observed in testing for changes in panel data that the effect of strong correlation between panels might make standard statistical procedures invalid (cf.\ Horv\'ath and Hu\v{s}kov\'a (2012)).  Bai and Zhang (2010) provide some economic examples when the common factor between the panels is strong so its effect must be built into testing and estimating procedures.  Our next result covers the case when the order of the correlation between the panels and the sizes of the changes are essentially the same. Since the contribution of the  $\eta_t$'s to the limit will not be negligible we need to specify the relation between the errors and the common factors:
\begin{assumption} \label{total}{\rm $\mbox{The sequences  } \{\eta_t, -\infty<t<\infty\}\mbox{ and }\{e_{i,t}, -\infty<t<\infty, 1\leq i \leq N\}  \mbox{ are } \mbox{independent.}$
}
\end{assumption}

Similarly to ${\mathcal U}_i(t)$, we introduce
\begin{displaymath}
{\mathcal V}(t)=\left\{
\begin{array}{ll}
\displaystyle \sum_{s=1}^t\eta_{s},\;\;&\mbox{if}\;\;t=1,2,3,\ldots
\vspace{.2 cm}\\
0,&\mbox{if}\;\;t=0
\vspace{.2 cm}\\
\displaystyle -\sum_{s=t}^{-1}\eta_{s},\;\;&\mbox{if}\;\;t=-1,-2,-3,\ldots,
\end{array}
\right.
\end{displaymath}
which  will be part of the limit distribution when \eqref{conv2del} holds or $\Sigma_{N,T}$ is proportional to $\Delta_{N,T}$. In all the other cases
we  assume the asymptotic normality of ${\mathcal V}(t)$:
\begin{assumption}\label{weak-v}{\rm
$$T^{-1/2}\sum_{t=1}^{\lf Tu\rf}\eta_t\;\;\stackrel{{\mathcal D}[0,1]}{\longrightarrow}\;\;W(u),\;\;\mbox{where } W \mbox{ is a Wiener process.}$$}
\end{assumption}

\medskip
\begin{theorem}\label{w-fac-mid} We assume that Assumptions \ref{b-1}--\ref{weak-v}, and \eqref{bar-tau} hold, and
\beq\label{sidelsame}
0<| {\frak s}|<\infty.
\eeq
(a) If \eqref{w-1} holds, then we have
\beq\label{f-1-m}
\frac{\Delta_{N,T}(\hat{t}_{N,T}-t_0)}{\sigma^2+{\frak s}^2}\;\;\;\stackrel{{\mathcal D}}{\to}\;\;\mbox{{\rm argmax}}_u \left\{W(u)-g_\theta (u)   \right\},
\eeq
where $W(u), -\infty<u<\infty$ is a two--sided Wiener process.\\
(b) If \eqref{conv2del} holds,
then we have
\beq\label{f-2-m}
\hat{t}_{N,T}-t_0\;\;\;\stackrel{{\mathcal D}}{\to}\;\;\mbox{{\rm argmax}}_t \{{\frak G}(t)+{\frak s}{\frak d}^{1/2}{\mathcal V}(t)-{\frak d}g_\theta(t)\},
\eeq
where ${\frak G}(t), t=0, \pm 1, \pm 2, \ldots $ is the  Gaussian process defined in Theorem \ref{w-fac}, independent of ${\mathcal V}(t), t=0, \pm 1, \pm 2, \ldots $
\end{theorem}

The effect of correlation between the panels is demonstrated in Theorem \ref{w-fac-mid}. The limit distribution in \eqref{f-1-m} remained the same as in \eqref{f-1} but the variance of the estimator increased by ${\frak s}^2$. The effect of the common factor is more transparent in \eqref{f-2-m} since an additional term appears in the limit which depends on the distribution of the common factors.
\medskip
\begin{theorem}\label{w-fac-strong} If Assumptions \ref{b-1}--\ref{weak-v}, and \eqref{bar-tau} hold, and
\beq\label{sidestrong}
| {\frak s}|=\infty,
\eeq
then we have
\beq\label{f-1-s}
\frac{\Delta^2_{N,T}}{\Sigma^2_{N,T}}(\hat{t}_{N,T}-t_0)\;\;\;\stackrel{{\mathcal D}}{\to}\;\;\mbox{{\rm argmax}}_u \left\{W(u)-g_\theta (u)   \right\},
\eeq
where $W(u), -\infty<u<\infty$ is a two--sided Wiener process.  
\end{theorem}

\medskip
Theorem \ref{w-fac-strong} covers the case when the limit distribution of the estimator for the time of change is completely determined by the common factors. The limit distribution in \eqref{f-1-s} is the same as in \eqref{f-1} and \eqref{f-1-m} but the rate of convergence is much slower. The effect of having several panels with changes is overshadowed by the strong influence of the common factor. For further results when the common factors dominate the limit distribution we refer to Kim (2011). \\

\section{Examples}\label{sec-exa}\setcounter{equation}{0}
In this section, we study some examples of error processes that satisfy the assumptions of Section \ref{sec-main}. We restrict our attention to establishing Assumption \ref{cond-1} for examples of possible model error sequences $e_{i,t}$'s, but the same sequences could be used for the common factors as well.\\

\begin{example}\label{exa-ind} {\rm Let $e_{i,t}, -\infty<t<\infty$ be independent, identically distributed random variables with $Ee_{i,0}=0$, $Ee^2_{i,0}=\sigma_i^2$ and $E|e_{i,0}|^\kappa<\infty.$ Due to independence we have that
\beq\label{exa-1}
U_{i,2}(t)=t\sigma_i^2\;\;\;\mbox{for all  } t=0, 1, 2\ldots
\eeq
  By the Rosenthal inequality (cf.\ Petrov (1995, p.\ 59)) we obtain  for all $t\geq 1$ that
$$
U_{i,\kappa}(t)\leq C\left\{tE|e_{i,0}|^\kappa+t^{\kappa/2}\sigma_i^{\kappa}\right\}\leq Ct^{\kappa/2}\{E|e_{i,0}|^\kappa+\sigma_i^{\kappa}\}
,
$$
where $C$ is an absolute constant, depending on $\kappa>2$ only. If the error terms in each panel are independent and identically distributed, then,  assuming $C_1\leq \sigma_i\leq C_2$ for all $1\leq i \leq N$,  Assumption \ref{cond-1}(i) holds; Assumption \ref{cond-1}(ii) is satisfied if
\beq\label{exa-2}
\frac{1}{N}\sum_{i=1}^N(E|e_{i,0}|^\kappa)^2=O(1)\;\;\;\mbox{with some}\;\;\kappa>4.
\eeq
If $\max_{1\leq i \leq N}E|e_{i,0}|^{\bar{\tau}}\leq C_0$ with some $\bar{\tau}>2$ and $C_0>0$, then Assumption \ref{cond-tau}(ii) is also fulfilled.}
\end{example}

\medskip

ARMA processes are very often used in classical time series analysis and our next example shows that stationary ARMA processes satisfy the basic assumptions of the first section. We consider the more general case of linear processes, which are investigated by Bai (2010), Kim (2011), and Horv\'ath and Hu\v{s}kov\'a (2012).

\begin{example}\label{exa-lin}{\rm  We assume that $\{\vare_{i,t}, -\infty <t<\infty\}$ are independent and identically distributed random variables with $E\vare_{i,0}=0$ and $E|\vare_{i,0}|^\kappa<\infty$ with some $\kappa>4$. The error terms $e_{i,t}$ form a linear process given by
$$
e_{i,t}=\sum_{\ell=0}^\infty c_{i,\ell}\vare_{i,t-\ell},
$$
where $\sup_\ell\ell^{-2-\alpha_i}|c_{i,\ell}|\leq C_i$ with some $C_i>0$ and $\alpha_i>0$.
By the Phillips and Solo (1992) representation we get
$$
\sum_{s=1}^t e_{i,s}-\bar{C}_i\sum_{s=1}^t\vare_{i,s}=\sum_{j=-\infty}^\infty\left( \sum_{k=1}^t\bar{c}_{k-j} \right)\vare_{i,j},
$$
where $\bar{C}_i=\sum_{\ell=0}^\infty c_{i,\ell}\neq 0, \bar{c}_{i,0}=c_{i,0}-\bar{C}_i, \bar{c}_{i,\ell}={c}_{i,\ell}$, if $\ell\geq 1$ and $\bar{c}_{i,\ell}=0$, if $\ell\leq -1.$ Minkowski's inequality and the discussion in Example \ref{exa-ind} yield that we need to choose $\sigma_i^2=\bar{C}_i^2E\vare ^2_{i,0}$ in Assumption \ref{cond-1}(i) and we also have  $U_{i,\kappa}(t)=O(t^{\kappa/2})$ and $|Ee_{i,0}e_{i,t}|=O(t^{-2-\alpha_i})$, as $t\to \infty$.
}\end{example}
\medskip

The next example assumes a martingale structure of the errors.

\begin{example}\label{exa-mart}{\rm  Let us assume that $\{e_{i,t}, -\infty<t<\infty\}$ is a stationary orthogonal martingale difference sequence with respect to some filtration with $Ee_{i,0}=0$, $Ee^2_{i,0}=\sigma_i^2$ and $E|e_{i,0}|^\kappa<\infty.$ Then \eqref{exa-1} as well as Assumption \ref{cond-tau} hold. By Li (2003) we also have
$$
U_{i,\kappa}(t)\leq Ct^{\kappa/2}E|e_{i,0}|^\kappa\;\;\;\mbox{with some constant  } C\;\;\mbox{depending only on  }\kappa,
$$
completing the proof of \eqref{exa-2}. Under assumption  $\max_{1\leq i \leq N}E|e_{i,0}|^{\bar{\tau}}\leq C_0$ with some $\bar{\tau}>2$ and $C_0>0$, we obtain Assumption \ref{cond-tau}.
}
\end{example}

\medskip
Since the early 1980's, ARCH, GARCH processes and their various extensions have become extremely popular  models in the analysis of macroeconomic and financial data. For a survey and detailed study of volatility models we refer to Francq and Zako\"{\i}an (2010). The next example shows that a large class of volatility processes satisfies the assumptions in Section \ref{sec-main}.\\
\begin{example}\label{exa-vol}{\rm We assume that $\{\vare_{i,t}, -\infty <t<\infty\}$ are independent and identically distributed random variables with $E\vare_{i,0}=0$ and $E\vare_{i,0}^2=1$. The error terms are defined by
\beq\label{garch}
e_{i,t}=h_{i,t}\vare_{i,t},
\eeq
where the volatility process $h_{i,t}>0$ is measurable with respect to the $\sigma$--algebra generated by $\vare_{i,s}, s\leq t-1$. Usually, $h_{i,t}$ is given by a recursion involving $e_{i,s}, h_{i,s}, s\leq t-1$. Francq and Zako\"{\i}an (2010) provides conditions for the existence of a stationary solution of \eqref{garch} in several models and establishes their basic properties. Assuming that $E|e_{i,0}|^\kappa<\infty $ with some $\kappa>4$,  $\{e_{i,t}, -\infty<t<\infty\}$ is a stationary orthogonal martingale satisfying the conditions in Example \ref{exa-mart}.
In case of the most popular GARCH$(p,q)$ model $h_{i,t}=\omega +\sum_{\ell=1}^p\alpha_\ell e_{i,t-\ell}^2+\sum_{\ell=1}^q\beta_j h_{i,t-j},$  $\omega>0, \alpha_\ell\geq 0, \beta_j\geq 0, 1\leq \ell \leq p, 1\leq j \leq q.$ The necessary and sufficient condition for the existence of the higher moments in case of GARCH(1,1) is given in Nelson (1990). He and Ter\"{a}svirta (1999), Ling and McAleer (2002) and Berkes et al.\ (2003) partially extends his results to the more general case. The existence of moments of augmented GARCH sequences are discussed in Carrasco and Chen (2002) and H\"{o}rmann (2008).
}
\end{example}

\medskip
Linear processes and the volatility models of Example \ref{exa-vol} are in the class of $m$--decomposable processes.
\begin{example}\label{exa-deco}{\rm We say the $e_{i,t}$ is a Bernoulli shift if it can be written as
$$
e_{i,t}=f_i(\vare_{i,t}, \vare_{i,t-1},\vare_{i,t-2}, \ldots )
$$
with some functional $f_i$, where $\{\vare_{i,t}, -\infty<t<\infty\}$ are independent and identically distributed random variables. The conditions of Section \ref{sec-main-res} are satisfied if the  Bernoulli shift is $m$--decomposable, i.e. if
$$
\sum_{m=1}^\infty( E|e_{i,t}-e^{(m)}_{i,t}|^{\kappa})^{1/\kappa}<\infty \;\;\;\mbox{with some   }\kappa>4,
$$
 where $e^{(m)}_{i,t}=f_i(\vare_{i,t}, \vare_{i,t-1},\ldots ,\vare_{i,t-m+1},\vare^*_{i,t-m}, \vare^*_{i,t-m-1}\ldots )$, and the $ \vare^*_{i,t}$'s   are independent copies of $\vare_{i,0}$, independent of $\vare_{i,t}, 1\leq i \leq N, -\infty <t<\infty$.  Berkes et al.\ (2011) proved that there is constant $C_i$ such that
$
U_{i,\kappa}(t)\leq C_i t^{\kappa/2},
$
 $|Ee_{i,0}e_{i,t}|\leq C_it^{-\kappa/2}$ and $U_{i,2}(t)=t\sigma_i^2+(t)$, as $t\to \infty$, with some $\sigma_i^2.$
They also provide several examples for $m$--decomposable Bernoulli shifts.
}
\end{example}

\medskip
\begin{example}\label{exa-mix}{\rm There is a well developed theory of partial sums of mixing random variables where the long run moments $U_{i,\kappa}(t)$ play a crucial role. It has been established under various conditions that $U_{i,2}(t)=t\sigma_i^2+o(t)$ and $U_{i,\kappa}(t)=O(t^{\kappa/2})$, as $t \to \infty.$ For surveys on mixing processes we refer to Bradley (2007) and Dedecker et al.\ (2007).
}
\end{example}

\medskip
\begin{example}\label{exa-gen}{\rm We assumed in Examples \ref{exa-lin} and \ref{exa-vol} that the innovations $\vare_{i,t}, -\infty <t<\infty$ are independent and identically distributed. However, this assumption can be replaced with the less restrictive requirement that $\vare_{i,t}, -\infty <t<\infty$ is a stationary sequence. Rosenthal--type inequalities for sums of functionals of stationary processes are developed in Wu (2002) and Merlev\'ede et al.\ (2006). These results can be used to establish Assumptions \ref{cond-1}, and \ref{cond-tau}.
}
\end{example}

\bigskip

\section{Estimation of norming sequences}\label{subest}

Theorems \ref{w-fac}--\ref{w-fac-strong} contain the limit distribution of $\hat{t}_{N,T}$ with different normalizations to show the effects of the sizes of changes and the loading factors. However, in case of finite $N$ and $T$ it is impossible to check which specific condition on the growths of $\Delta_{N,T}$ and $\Sigma_{N,T}$ holds, so we require norming sequences that would work in all possible cases. Let
$$
\Xi_{N,T}=\sum_{i=1}^N\sigma_i^2\delta_i^2+ \displaystyle\Sigma^2_{N,T}.
$$
Under the conditions of Theorems \ref{w-fac}--\ref{w-fac-strong} we have that
\beq\label{joint}
\frac{\Delta^2_{N,T}}{\Xi_{N,T}}(\hat{t}_{N,T}-t_0)\;\;\;\;\mbox{converges in distribution}.
\eeq
The limit distribution in \eqref{joint} is $\mbox{\rm argmax}_u\{W(u)-g_\theta(u)\}$ except in the special cases of \eqref{f-2} and \eqref{f-2-m}. In these cases, the limit distribution is the argmax of a process defined on integers. The limits in \eqref{f-2} and  \eqref{f-2-m}
depend on the distributions of $\{X_{i,t}, 1\leq i \leq N, 1\leq t \leq T\}$. If ${\frak d}$ is close to 0 in \eqref{f-2} or \eqref{f-2-m} the limiting distributions can be well approximated with $\mbox{\rm argmax}_u(W(u)-g_\theta(u))$.  Bai (2010) provides numerical evidence that $\mbox{\rm argmax}_u\{W(u)-g_\theta(u)\}$ gives a reasonable approximation for the limit in \eqref{f-2} when ${\frak u}(s,t)=\min(t,s)$. Hence we recommend that $\mbox{\rm argmax}_u(W(u)-g_\theta(u))$ can be used as the limit in \eqref{joint} in practice.\\

The limit result in \eqref{joint} can only be used  for hypothesis testing or confidence intervals if the norming factor can be consistently estimated from the sample. We estimate $\Delta_{N,T}$ with
$$
\hat{\Delta}_{N,T}=\sum_{i=1}^N\left(\frac{1}{\hat{t}_{N,T}}\sum_{1\leq t \leq \hat{t}_{N,T}}X_{i,t}-\frac{1}{T-\hat{t}_{N,T}}\sum_{\hat{t}_{N,T}<t\leq T}X_{i,t}
\right)^2.
$$
It is more difficult to estimate $\Xi_{N,T}$. Let
\beq\label{u-def}
U_N(t)=
\sum_{i=1}^N\left(S_i(t)-\frac{t}{T}S_i(T)\right)^2,
\eeq
$$
\hat{r}_{N,T}(t)=-t\frac{T-\hat{t}_{N,T}}{T}+(t-\hat{t}_{N,T})I\{t>\hat{t}_{N,T}\},
$$
and
\beq\label{hat-r}
\hat{r}_{N,T}=\hat{r}_{N,T}(\hat{t}_{N,T})=-\hat{t}_{N,T}\frac{T-\hat{t}_{N,T}}{T}.
\eeq
The estimator for $\Xi_{N,T}$ is defined as
\begin{align*}
\hat{\Xi}_{N,T}
=\frac{1}{2(M_2-M_1)}&\sum_{M_1<|v| \leq M_2}\frac{1}{4|v|\hat{r}^2_{N,T}}\biggl(U_N(\hat{t}_{N,T}+v)-U_N(\hat{t}_{N,T})\\
&-\hat{\Delta}_{N,T}(\hat{r}^2_{N,T}(\hat{t}_{N,T}+v)- \hat{r}_{N,T}^2 )
\biggl)^2.
\end{align*}

\begin{theorem}\label{th-cons} (i) We assume that the conditions of Theorem \ref{w-fac} or \ref{w-fac-mid} are satisfied. If $M_1<M_2$, $M_1\to \infty$ and $M_2/T\to 0$, then
\beq\label{d-cons}
\frac{\hat{\Delta}_{N,T}}{{\Delta}_{N,T}}\;\;\;\stackrel{P}{\to}\;\;1
\eeq

and

\beq\label{xi-const}
\frac{\hat{\Xi}_{N,T}}{\Xi_{N,T}}\;\;\stackrel{P}{\to}\;\;1.
\eeq
\noindent
(ii)  We assume that the conditions of Theorem  \ref{w-fac-strong} are satisfied. If $M_1<M_2$, $M_1\to \infty$ and
$$
M_2/\min( T, \Sigma_{N,T}^2/\Delta^2_{N,T})\to 0,
$$
then \eqref{d-cons} and \eqref{xi-const} hold.

\end{theorem}

Since  \eqref{d-cons} and \eqref{xi-const} hold, the limit result in \eqref{joint} remains  true when the norming is replaced with the corresponding estimators, i.e.
\beq\label{joint-est}
\frac{\hat{\Delta}^2_{N,T}}{\hat{\Xi}_{N,T}}(\hat{t}_{N,T}-t_0)\;\;\;\;\mbox{converges in distribution}.
\eeq

If the interaction between the panels is small. i.e.\ $\Sigma^2_{N,T}/\sum_{i=1}^N\sigma_i^2\delta_i^2\to 0$, as $N,T\to \infty$, then we need to estimate only $\Delta_{N,T}$ and  $\sum_{i=1}^N\sigma_i^2\delta_i^2$. Using $\hat{\sigma}^2_i, 1\leq i \leq N$, the long run variance estimators  for $\sigma_i^2$, a possible estimator for $\sum_{i=1}^N\sigma_i^2\delta_i^2$ is $\sum_{i=1}^N\hat{\sigma}_i^2[\sum_{1\leq s \leq \hat{t}_{N,T}}X_{i,s}/\hat{t}_{N,T}-\sum_{\hat{t}_{N,T}<s\leq T}X_{i,s}/(T-\hat{t}_{N,T})]^2$.

\section{Simulations, and data examples}\label{sec-simu}\setcounter{equation}{0}

\subsection{Simulations}\label{subsec-simu}
Using the estimators defined in Section \ref{subest}, we computed the empirical percentages when the variable defined in (3.5)  is below the asymptotic quantiles for several  $N$ and $T$. We considered the case when there is no interaction between the panels, i.e.\ $\gamma_i=0$ and also cases with  correlations. We tried various values of $t_0=\lf T\theta\rf$. We observed that the applicability of the limit results presented in Section \ref{sec-main} does not depend on $\theta$. In Tables \ref{tab-1} and \ref{tab-2} we used $\theta=1/2$ and they are based on 1,000 repetitions. The $90^{\mbox{th}}$, $95^{\mbox{th}}$ and $99^{\mbox{th}}$ percentiles of the distribution of
$\mbox{argmax}_{u}(W(u)-g_{1/2}(u))$ are $4.70$, $7.69$ and $15.89$, respectively.  Table \ref{tab-1} illustrates that $\Delta_{N,T}$ must be small to use Theorem \ref{w-fac}(a) and the approximation improves when $T$ increases. In case of larger $\Delta_{N,T}$ we are under the conditions of Theorem \ref{rate-con-1} and the distribution of $\hat{t}_{N,T}$ is more concentrated than we would get from the asymptotics in Theorem \ref{w-fac}(a).

\begin{table}[htdp]
\caption{Empirical percentages of ${\Delta^2_{N,T}}(\hat{t}_{N,T}-t_0)/{\Xi_{N,T}}$ below the asymptotic  quantiles  for  various values of  $N, T, \delta_i$ when $\gamma_i=0$ for all $1\leq i \leq N$.}
\begin{center}
\begin{tabular}{cc|ccc}
N/T  & $\delta_i$ 
 & 90\%   &  95\%   &99\%       \\
\hline
25/100 & 0.150 & 88.7\% & 94.7\% & 100\%\\
25/250 & 0.100 & 87.6\% & 94.4\% & 99.8\% \\
25/500 & 0.060 & 89.9\% & 96.4\% & 100\% \\
50/100& 0.100 &  89.1\%  &  97.1\%   &  100\%\\
50/250 & 0.070 &  88.1\%  &  95.4\%   &  100\%\\
50/500 & 0.050 & 86.5\%    & 94.9\%    & 100\%\\
100/100 & 0.085 & 86.1\%    &94.5\%   &  100\%\\
100/250 & 0.055 &  86.9\% &  94.3\%  & 99.9\%\\
100/500 & 0.035 & 88.5\%  &  96.5\%  & 100\%\\
\hline
\end{tabular}
\end{center}
\label{tab-1}
\end{table}%

\begin{table}[htdp]
\caption{Empirical percentages of ${\Delta^2_{N,T}}(\hat{t}_{N,T}-t_0)/{\Xi_{N,T}}$ below the asymptotic  quantiles  for  various values of  $N, T, \delta_i$ when $\gamma_i=0.03$ for all $1\leq i \leq N$.}
\begin{center}
\begin{tabular}{cc|ccc}
N/T  & $\delta_i$ 
 & 90\%   &  95\%   &    99\%       \\
\hline
25/100 & 0.150 & 86.9\% & 94.9\% & 100\% \\
25/250 & 0.100 & 90.7\% & 95.6\% & 100\% \\
25/500 & 0.060 & 89.2\% & 97.2\% & 100\% \\
50/100& 0.100 &  87.6\%  &  95.6\%   &  100\%\\
50/250 & 0.070 &  91.0\%  &  96.4\%   &  100\%\\
50/500 & 0.050 & 88.3\%	& 96.3\%	& 100\%\\
100/100 & 0.085 & 89.4\%	&95.9\%   &  100\%\\
100/250 & 0.055 &  85.1\% &  94.0\%  & 99.9\%\\
100/500 & 0.035 & 90.1\%  &  96.5\%  & 100\%\\
\hline
\end{tabular}
\end{center}
\label{tab-2}
\end{table}%

Our comments also  holds when interaction between the panels is allowed as illustrated by Table \ref{tab-2} but, due to the dependence, larger $T$ is needed to use limit results.\\

Figure \ref{sim12} shows  that the density function of the limit follows the shape of the histogram of $\hat{t}_{N,T}$ closely.
\begin{figure}
\centering
\caption{The histogram of $\hat{t}_{100,500}$ with $\delta_i=0.07, \gamma_i=0, 1\leq i\leq 100$ (left panel) and the histogram of $\hat{t}_{50, 100}$ with $\delta_i=0.1, \gamma_i=0.03, 1\leq i \leq 50$ (right panel) and the density of the limiting random variable}
\mbox{\subfigure{\includegraphics[width=3.3in]{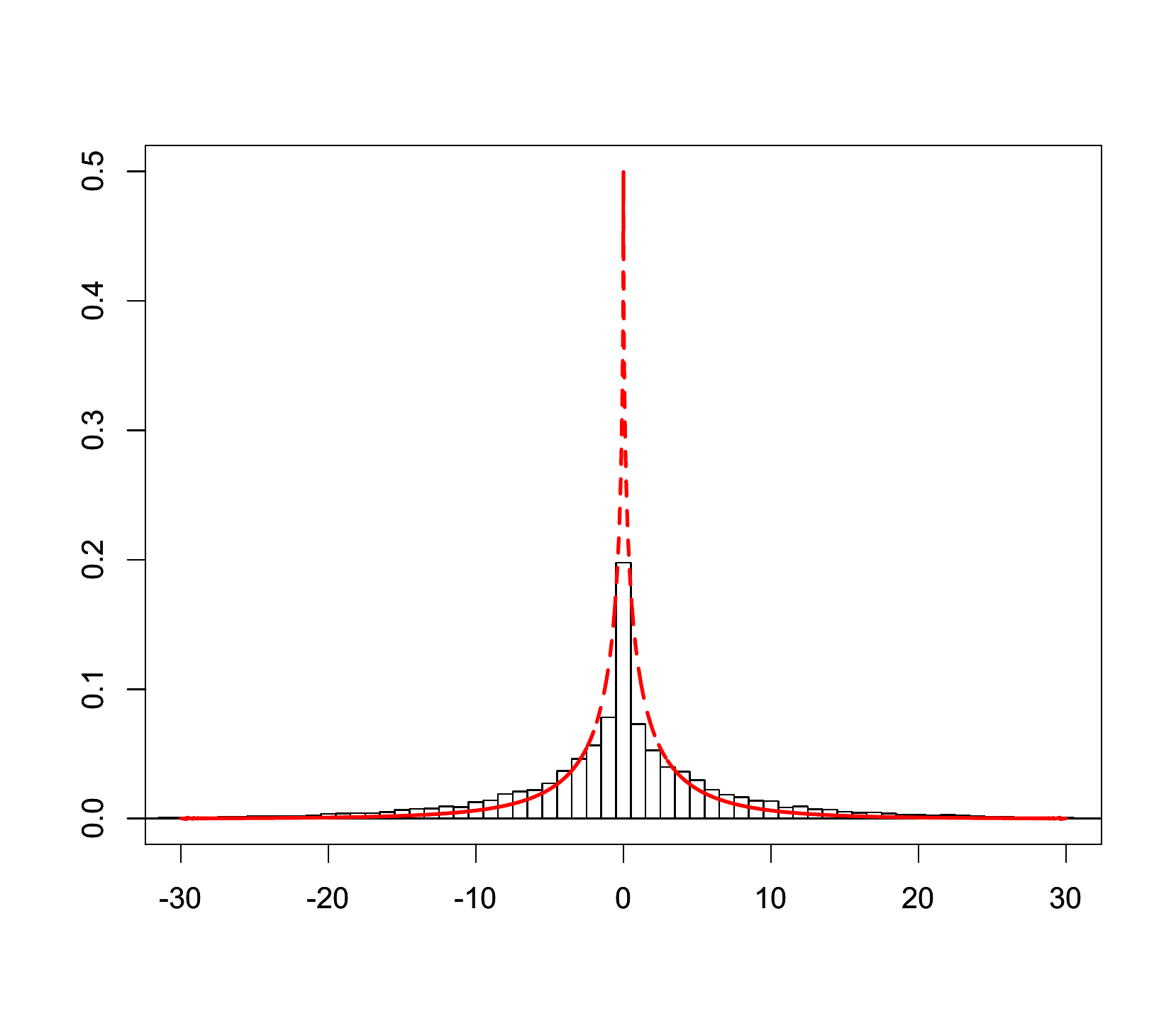}}\quad
\subfigure{\includegraphics[width=3.3in]{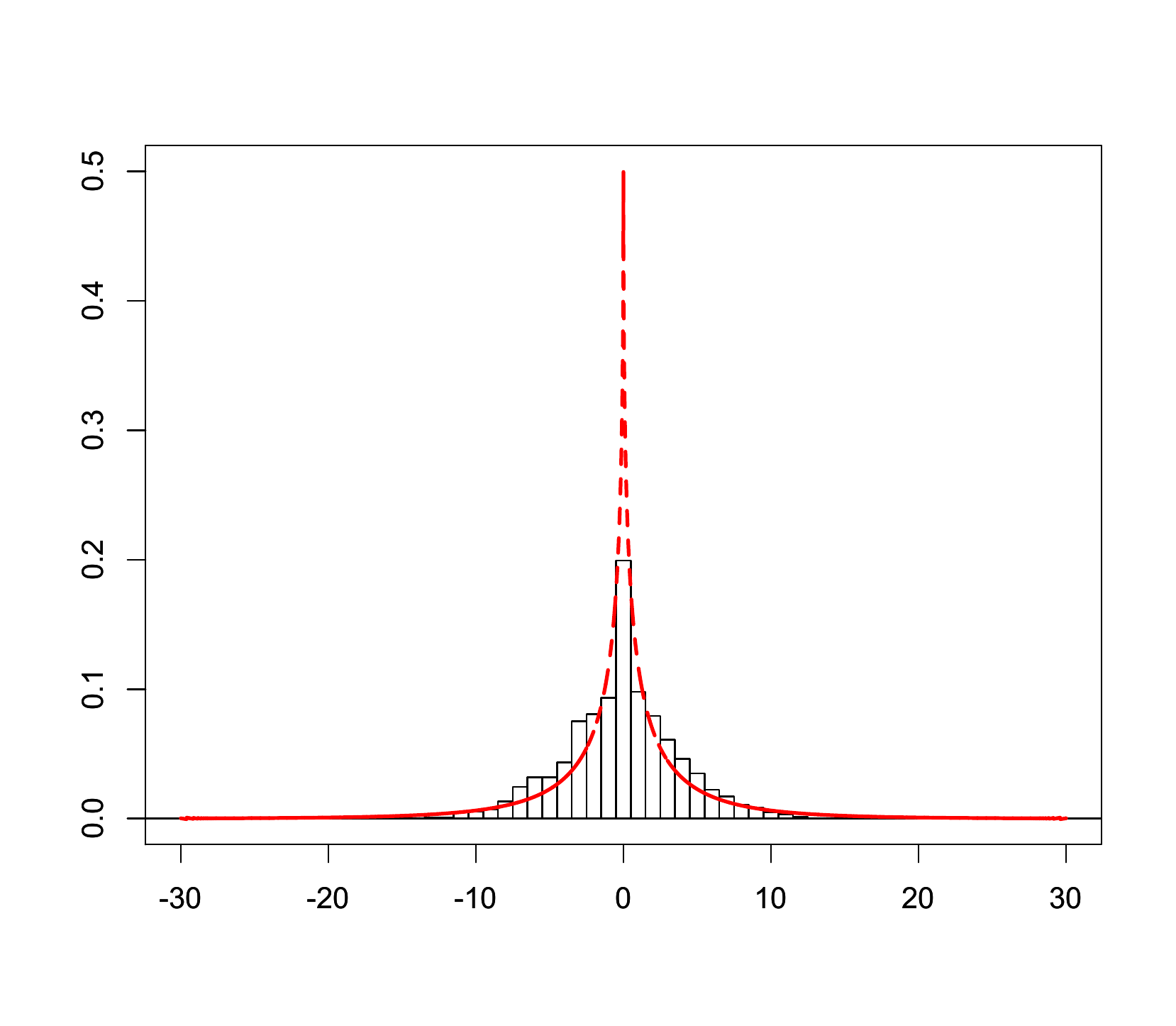} }}
\label{sim12}
\end{figure}

\subsection{Applications}\label{subsec-appl}
\begin{figure}
\centering
\caption{The graphs  exchange rates, 1=UK, 2=SI, 3=CA, 4=SW  (left panel); 1=DN, 2=NO, 3=SD (right panel) with respect to the US dollar.}
\mbox{\subfigure{\includegraphics[width=3.3in]{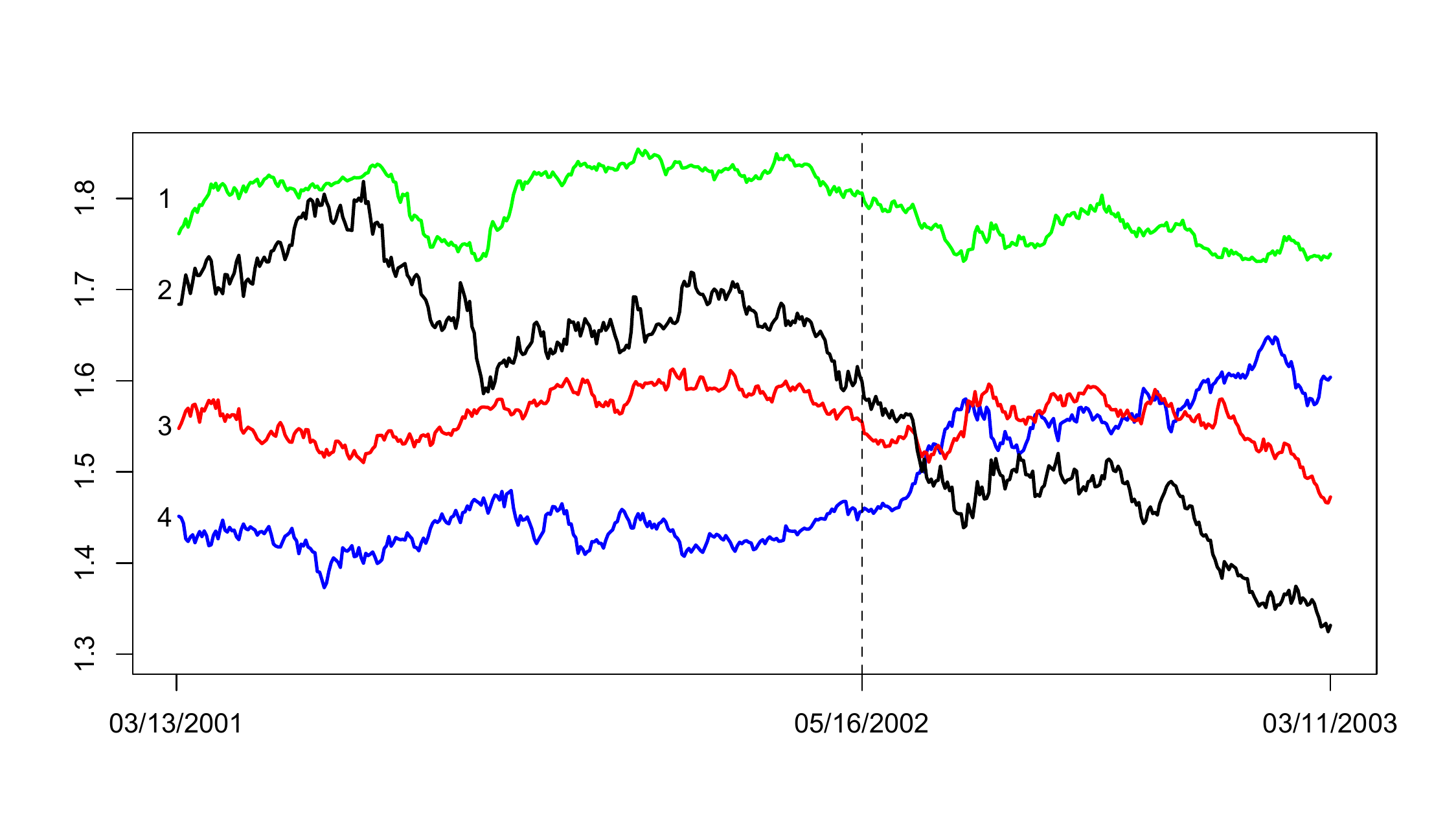}}\quad
\subfigure{\includegraphics[width=3.3in]{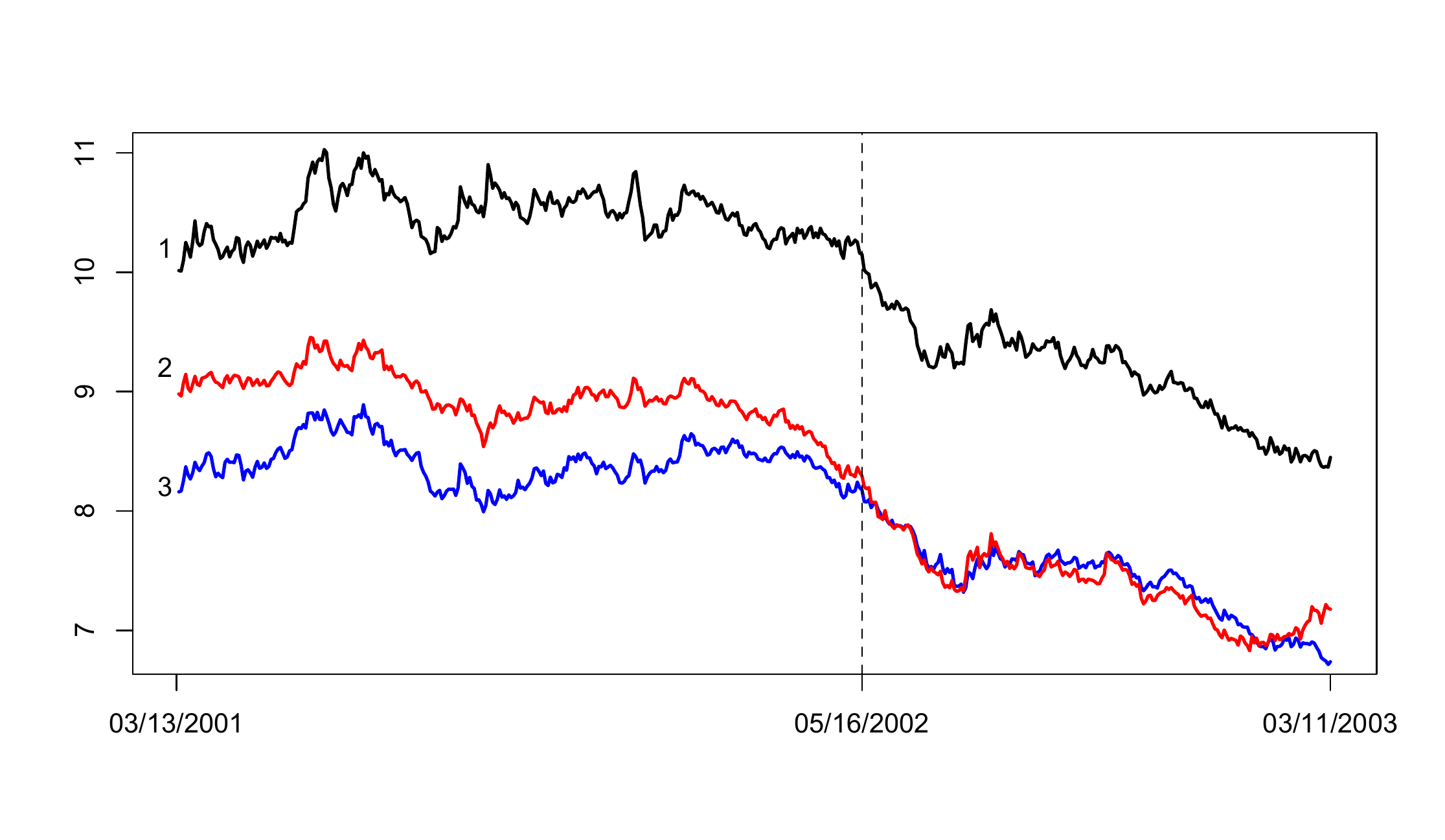} }}
\label{exam12}
\end{figure}

\begin{figure}
\centering
\caption{The graphs of relative exchange rates, 1=UK, 2=SI, 3=CA, 4=SW  (left panel); 1=DN, 2=NO, 3=SD (right panel) with respect to the US dollar.}
\mbox{\subfigure{\includegraphics[width=3.3in]{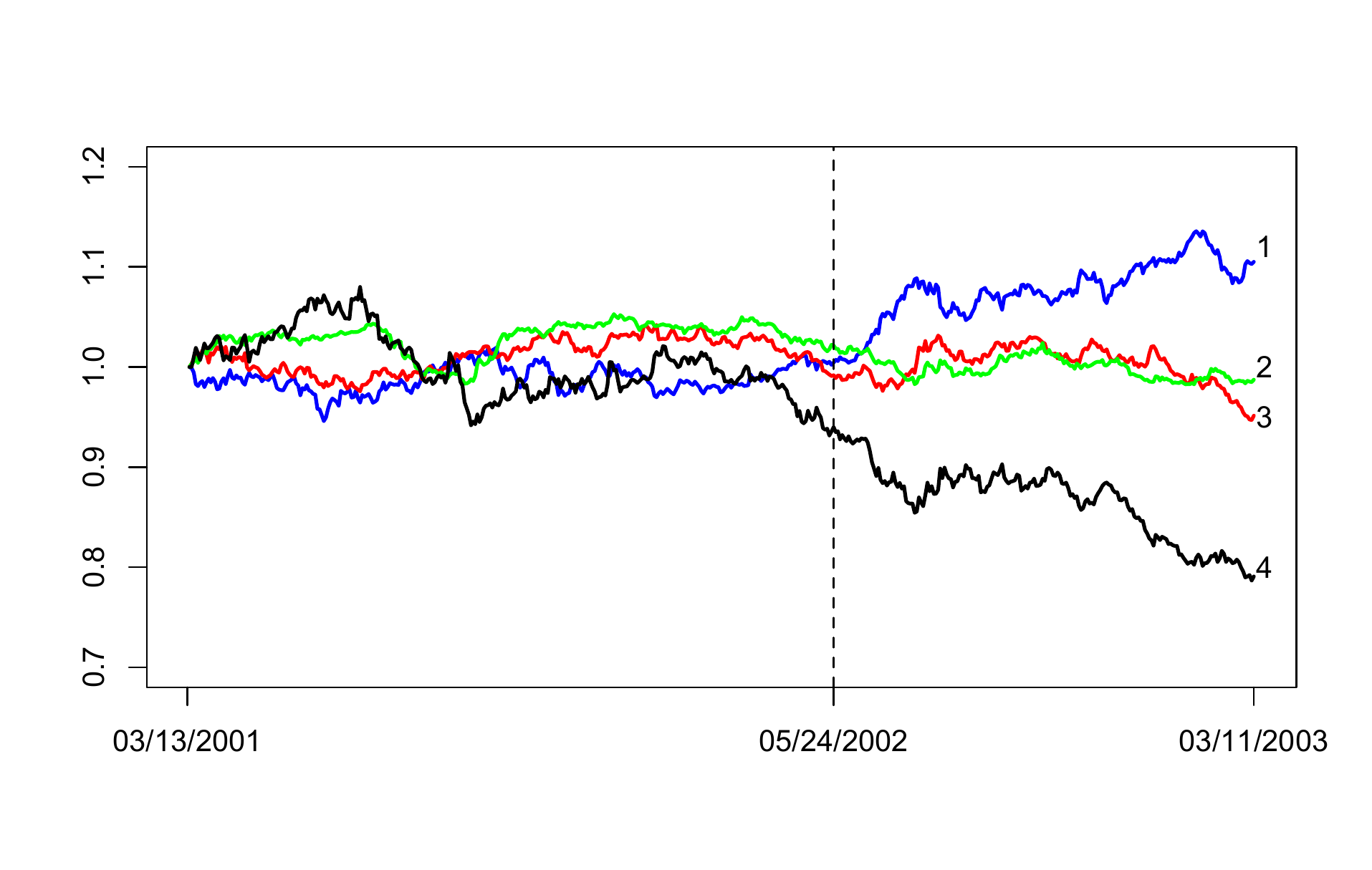}}\quad
\subfigure{\includegraphics[width=3.3in]{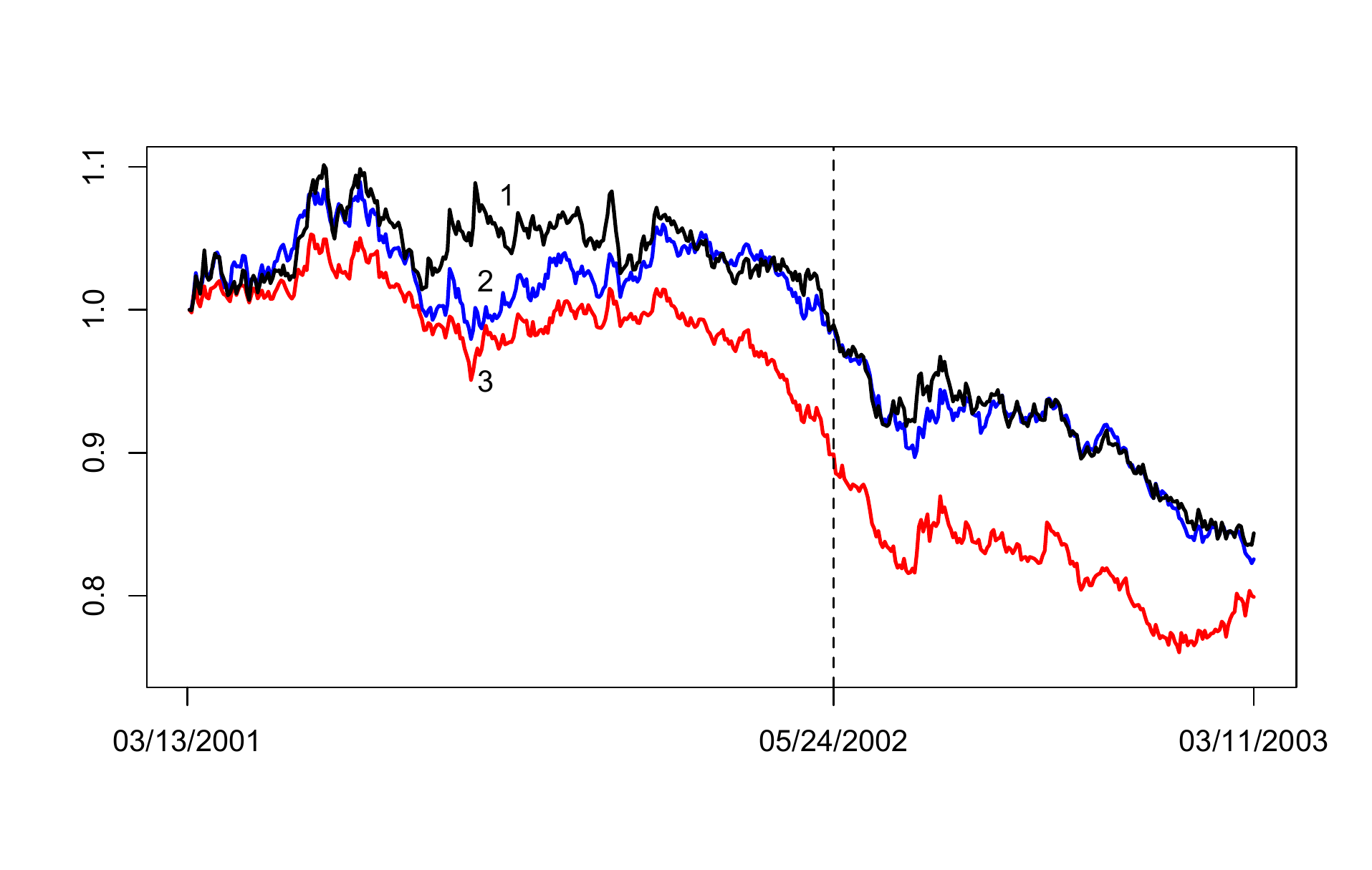} }}
\label{scaled-exam12}
\end{figure}

In the first example we consider the exchange rates between the US dollar and 23 other currencies. The data can be found at the website {{\tt www.federalreserve.gov\\
/releases/h10/hist/}}.  Figure \ref{exam12} contains the graphs of the exchange rates between the United Kingdom (UK), Canada (CA), Singapore (SI), Switzerland (SW), Denmark (DN), Norway (NO) and Sweden (SD). In our study we used the time period 03/13/2001--03/11/2003 so we have $N=23$ panels and each panel has $T=500$ observations. Using the testing method in Horv\'ath and Hu\v{s}kov\'a (2013) the no change in the mean of the panels null hypothesis is rejected. The estimated time of change is $\hat{t}_{23, 500}=297$ so the change is indicated  on 05/16/2002. We also constructed confidence intervals using  \eqref{joint-est}.  Since $\hat{\Delta}_{23, 500}$ is very large, the 90\%, 95\% and 99\% confidence intervals contain only a single  element, $\hat{t}_{23, 500}=297$, i.e.\ the conditions of Theorem \ref{rate-con-1} hold in this case. It is clear from Figure \ref{exam12} that the exchange rates are between 1.3 and 11, so if the same proportional change occurs   in a panel with high values, this change will give a very large $\delta^2$ compared to the other panels. Hence a single panel can disproportionately contribute to $\hat{\Delta}_{N,T}$. To overcome this problem we rescaled the observations in each panel with the first observation, i.e.\ with the exchange rate on 03/13/2001.  Figure \ref{scaled-exam12} contains the graphs of the relative changes  in exchange rates with respect to the US dollar for the same countries as in Figure \ref{exam12}. We repeated our analysis  for the relative changes (rescaled) in the exchange rates with respect to the US dollar, resulting in $\hat{t}_{23, 500}=303$ which corresponds to 05/24/2002. In the definition of $\hat{\Xi}_{23, 500}$ we used $M_2-M_1\approx \lf T^{1/2}/\hat{\Delta}_{23, 500}\rf$. The 90\%, 95\% and 99\% confidence intervals are $[292, 330], [287, 341]$ and $[274, 371]$. Note that all these confidence intervals contain 297 which was obtained for the non--scaled exchange rates. Some of the graphs show a linear trend after the time of change instead of changing to an other constant mean. The limit distribution of the time of change is local, i.e.\ it is determined by the observations in the neighborhood of $t_0$. Replacing the linear trend with an average value close to $t_0$ can justify the asymptotic validity of the confidence intervals.  Also, after the change point was found  and the  means of the corresponding segments were removed, the stationarity of the residuals could not be rejected.\\
 The exchange rates data and the scaled exchange rates  contain further level shifts. Using the binary segmentation method one can divide the data into homogenous segments and provide confidence intervals for the time of changes. Sato (2013) investigates the number and the location of the changes in daily log returns in 30 currency pairs between 04/01/2001 and 30/12/2011 and points out that the study of the individual pairs might not find all the changes in the exchange rates mechanism.  \\

\begin{figure}
\centering \caption{The graphs of the log differences of the GDP/capita for 113 countries between 1961 and 2012.}\label{GDPcurv}
\includegraphics[scale=.55, width=5.5in,angle=0]{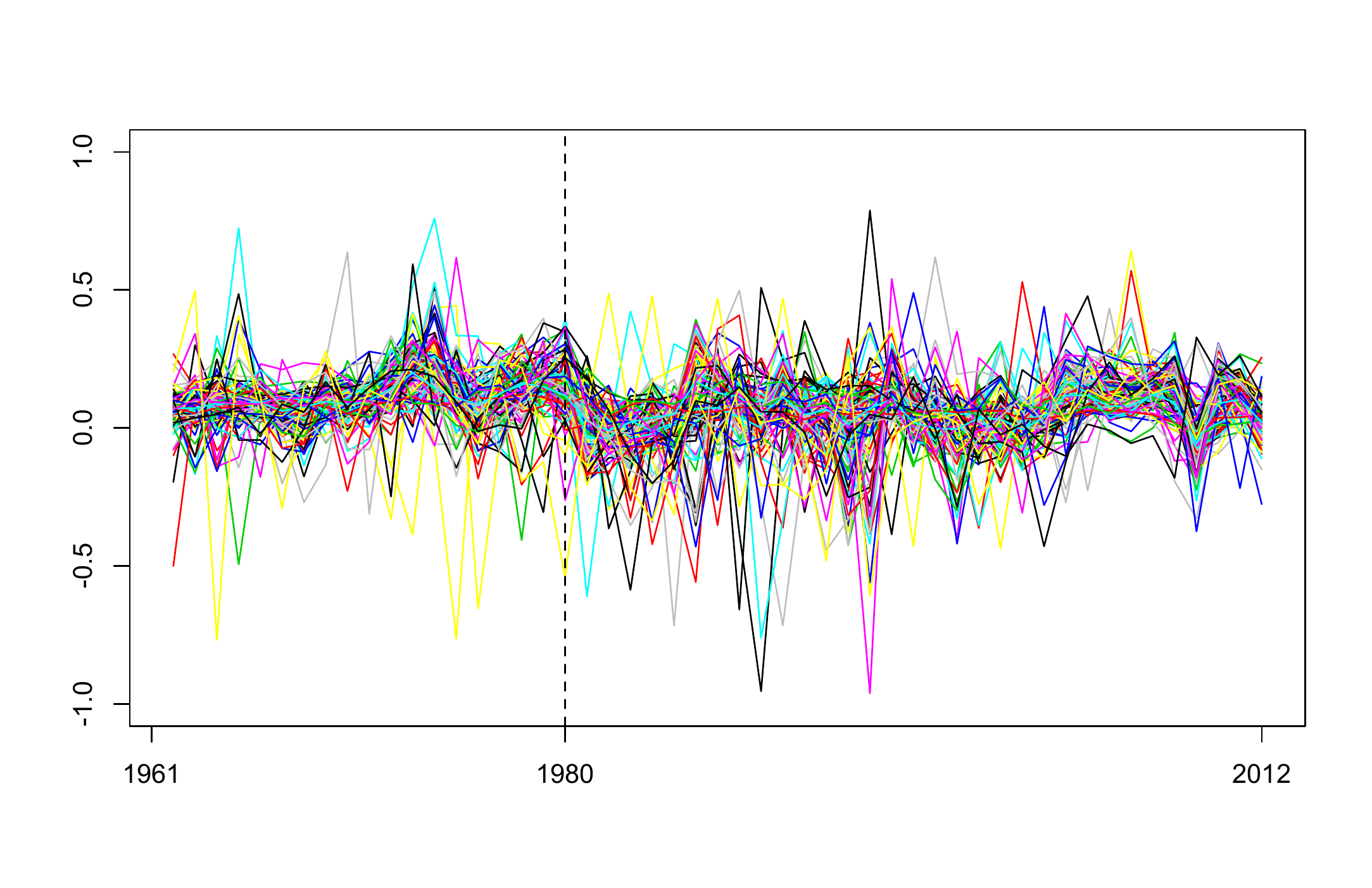}
\end{figure}

\medskip
In the second example we compared the GDP/capita of $N=113$ countries. The data can be found at the website {\tt www://data.worldbank.org/indicator/NY.GDP.MKTP.CD}.  The data are recorded in current US dollar. We removed some countries from the data set due to large number of missing values, so we used $N=113$ panels with $T=51$ covering the time period 1961--2012. To achieve stationarity of the errors we transformed the data by taking $\log$ differences. The graphs of the $\log$ transformed GDP's are exhibited in Figure \ref{GDPcurv}. We used the CUSUM test of Horv\'ath and Hu\v{s}kov\'a (2013) to test the stability of the means of the panels which was rejected at very high significance level. The estimator for the time of change is $\hat{t}_{113, 51}=19$ which corresponds to 1979/1980. Applying the limit result in \eqref{joint-est}, $[16, 21]$ is the asymptotic 90\% and 95\% confidence interval, while the 99\% confidence interval is $[12, 22]$.\\

\section{Conclusion}

We established the first and second order asymptotic properties of a CUSUM estimator of the time of change in the mean of panel data. Our results are derived under long run moment conditions, which serve to extend the asymptotic theory to a broader family of error processes than had been previously considered in the literature, and we provided an in depth study on how the rates of divergence of the sizes of changes and the common factor loadings are manifested in the asymptotic behavior of the test statistic. Our results were demonstrated with a Monte Carlo simulation study, and we considered application to two real data sets.

\subsection{Acknowledgements}
 We would like to thank three reviewers as well as the co--editor, Liangjun Su, for their helpful comments that lead to substantial improvements to this paper. This project was supported by NSF grant DMS 1305858 and grants GA\v{C}R 15-096635, GA\v{C}R 201/12/1277.

\bigskip

\noindent
{\bf REFERENCES}\\

\smallskip

\noindent
Andrews, D.\ W.\ K. (2003) End-of-sample instability tests. {\it Econometrica}{ 71}, 1661--1694.
Atak, A., O.\ Linton  \& Z.\ Xiao  (2011)  A semiparametric panel model for unbalanced data with application to climate change in the United Kingdom. {\it  Journal of Econometrics} { 164}, 92--115.\\
 Arellano, M. (2004) {\it Panel Data Econometrics.}   Oxford University Press.\\ 
 Aue, A.\ \&  L.\ Horv\'ath (2013) Structural breaks in time series. {\it Journal of Time Series Analysis} {23}, 1--16.\\
Bai, J. (2010)  Common breaks in means and variances for panel data.  {\it Journal of Econometrics}{ 157}, 78--92.\\
Bai, J.\ \&  J.L.\ Carrioni--i--Silvestre (2009) Structural changes, common stochastic trends, and unit roots in panel data.{\it The Review of Economic Studies} { 76}, 471--501.\\
Bai, Y.\ \& J.\ Zhang  (2010) Solving the Feldstein--Horioka puzzle with financial frictions. {\it Econometrica} { 78}, 603--632.\\
Baltagi, B.H.: {\it Econometric Analysis of Panel Data} 5th Edition, Wiley, New York, 2013.\\
 Baltagi, B.H.,\  C.\ Kao \&  L.\ Liu (2012) Estimation and identification of change points in panel models with nonstationary or stationary regressors and error terms. Preprint.\\
 Berkes, I.,\ S.\ H\"{o}rmann  \& J.\ Schauer  (2011). Split invariance principles for stationary processes. {\it The Annals of Probability} { 39}, 2441-2473.\\
 Berkes, I.,\ L.\ Horv\'ath \&  P.\ Kokoszka (2003)  GARCH processes: structure and estimation. {\it Bernoulli} { 9}, 201--227.\\
Bhattacharya, P.K.\ \&  P.J.\ Brockwell (1976) The minimum of an additive process with applications to
signal estimation and storage theory. {\it Zeitschrift f\"ur Wahrscheinlichkeitstheorie verwandte\ Gebiete} { 37}, 51--75.\\
Billingsley, P. (1968) {\it Convergence of Probability Measures.}
  Wiley.\\ 
Bradley, R.C.  (2007) {\it Introduction to Strong Mixing Conditions.} Vol.\ 1--3, Kendrick Press, Heber City, UT.\\
Brodsky, B.E.\ \& B.\ Darkhovskii (2000) {\it Non--Parametric Statistical Diagnosis.} Kluwer.\\ 
 Carrasco, M.\ \&  X.\ Chen (2002) Mixing and moment proprties of various GARCH and stochastic volatility models. {\it Econometric Theory} { 18}, 17--39.\\
 Chan, J.,\  L.\ Horv\'ath  \& M.\ Hu\v{s}kov\'a (2013)  Darling--Erd\H{os} limit results for change--point detetction in panel data. {\it Journal of Statistical Planning and Inference} { 143}, 955--970.\\
 Cs\"org\H{o}, M.\ \&  L.\ Horv\'ath (1997) {\it Limit Theorems in Change--Point Analysis.} Wiley. \\
 Cs\"org\H{o}, M.\ \&  P.\ R\'ev\'esz (1981) {\it Strong Approximations in Probability and Statistics.} Academic Press.\\ 
Dedecker, I.,\  P.\ Doukhan, G.\ Lang, J.R.R.\ Le\'on, S.\  Louhichi \& C.\ Prieur  (2007) {\it Weak Dependence with Examples and Applications.} Lecture Notes in Statistics, Springer, Berlin. \\
Ferger, D. (1994)  Change--point estimators in case of small disorders. {\it Journal of Statistical Planning and Inference} {40}, 33--49.\\
Francq, C.\ \& J.-M.\ Zakoian  (2010) {\it GARCH Models.} Wiley.\\ 
Frees, E.W. (2004) {\it Longitudinal and Panel Data: Analysis and Applications in the Social Sciences.}  Cambridge University Press.\\ 
 He, C.\ \&  T.\ Ter\"{a}svirta (1999) Properties of moments of a family of GARCH processes. {\it Journal of Econometrics} { 92}(1999), 173--192.\\
H\"ormann, S. (2008)  Augmented GARCH sequences: dependence structure and asymptotics. {\it Bernoulli}  {14}(2008), 543--561.\\
Horv\'ath, L.\ \&  M.\  Hu\v{s}kov\'a  (2012)  Change--point detection in panel data. {\it Journal of Time Series Analysis} {33}, 631--648.\\
 Horv\'ath, L.\ \&  G.\ Rice (2014)  Extensions of some classical methods in change point analysis (with discussions)   {\it Test} { 23}, 219--290.\\
 Hsiao, C. (2003)  {\it Analysis of Panel Data.} Second Edition. Cambridge University Press. \\
Hsiao, C. (2007)  Panel data analysis--advantages and challenges. {\it Test} {16}, 1--22.\\
Im, K.S.,  J.\ Lee \&  M.\ Tieslau (2005) Panel LM unit root test with level shifts. {\it Oxford Bulletin of Economics and Statistics} { 67}, 393--419.\\
Joseph, L.\ \& D.B.\ Wolfson  (1992) Estimation in multi--path change--point problems. {\it Communications in Statistics--Theory and Methods} { 21}, 897--913.\\
Joseph, L.\ \& D.B.\  Wolfson (1993)  Maximum likelihood estimation in the multi--path changepoint problem. {\it Annals of the Institute of Statistical Mathematics} { 45} , 511--530.\\
Kao, C.,\  L.\ Trapani \&  G.\ Urga (2012) Asymptotics for panel models with common shocks. {\it Econometric Reviews} { 31}, 390--439.\\
Kao, C.,\  L.\ Trapani \&  G.\ Urga (2014) Testing for breaks in cointegrated panels. {\it Econometric Reviews} To appear.\\
 Kim, D. (2011)  Estimating a common deterministic time trend break in large panels with cross sectional dependence {\it Journal of Econometrics} {164}, 310--330.\\
 Kim, D. (2014)  Common breaks in time trends for large panel data with a factor structure. {\it Econometrics Journal} In press.\\ 
Li, D.,\ J.\ Qian, \& L.\ Su (2014) Panel Data Models with Interactive Fixed Effects and Multiple Structrual Breaks. {\it Under Revision}\\
Qian, J.,\ \& L.\ Su (2014) Shrinkage Esimation of Common Breaks in Panel Data Models via Adaptive Group Fused Lasso. Working Paper, Singapore Management University.
Li, Y. (2003) A martingale inequality and large deviations. {\it Statistics \& Probability Letters} {62}, 317--321.\\
Ling, S.\ \&  M.\ McAleer (2002)  Necessary and sufficient moment conditions for GARCH(r,s) and asymmetric GARCH(r,s) model. {\it Econometric Theory} { 18}, 722--729.\\
Merlev\'ede, F.,\ M.\ Peligrad  \&  S.\ Utev  (2006)  Recent advances in invariance principles for stationary sequences. {Probability Surveys} { 3}, 1--36.\\
 M\'oricz, F.\ \, R.\ Serfling  \&  W.\ Stout (1982)  Moment and probability bounds with quasi-superadditive structure for the maximal partial sum. {\it Annals of Probability} {10}, 1032--1040.\\
 Nelson, D.B.  (1990)  Stationarity and persistance in the GARCH(1,1) model. {\it Econometric Theory} {6}, 318--334.\\
Petrov, V.V. (1995)  {\it Limit Theorems of Probability Theory}, Clarendon Press.\\ 
 Phillips, P.C.P.\ \&  V.\ Solo (1992)  Asymptotics    for linear processes. {\it Annals of Statistics}  {20}, 971--1001.\\
Qian, J.\ \& L.\ Su (2014) Shrinkage Estimation of Common Breaks in Panel Data Models via Adaptive Group Fused Lasso. {\it Under Revision} \\
Sato, A-H. (2013)  Recursive segmentation procedure based on the Akaike information criterion test. {\it 2013 IEEE 37th Annual Computer Software and Applications Conference}, 226--233.\\
Wooldridge, J.M. (2010) {\it Econometric Analysis of Cross Section and Panel Data.}  Second Edition, MIT Press.\\ 
Wu, W.B. (2002) Central limit theorems for functionals of linear processes and their applications. {\it Statistica Sinica} {12}, 635--649.\\

\appendix
\section{Proofs of Theorems \ref{rate-con-1}--\ref{w-fac-strong} and Remarks \ref{rem-bai} and \ref{rem-bai-cor}}\label{proof}\setcounter{equation}{0}
Throughout the proofs $c$ denotes unimportant constants whose values might change from line to line.
Using \eqref{model} we have
$$
S_i(t)-\frac{t}{T}S_i(T)=Q_i(t)+\gamma_iV(t)+\delta_ir(t),
$$
where
\beq\label{q-v-def}
Q_i(t)=\sum_{s=1}^te_{i,s}-\frac{t}{T}\sum_{s=1}^Te_{i,s},\;\;\;\;\;V(t)=\sum_{s=1}^t\eta_s-\frac{t}{T}\sum_{s=1}^T\eta_s
\eeq
and
\beq\label{r-def}
 r(t)=-t\frac{T-t_0}{T}+(t-t_0)I\{t>t_0\}.
\eeq
Hence  we have
\begin{align}\label{summa}
\left(S_i(t)-\frac{t}{T}S_i(T)\right)^2&=Q_i^2(t)+\gamma_i^2V^2(t)+\delta_i^2r^2(t)+2V(t)\gamma_iQ_i(t)+
2r(t)\delta_iQ_i(t)\\
&\hspace{1cm}+2V(t)r(t)\gamma_i\delta_i.\notag
\end{align}
 Let $0<\alpha<\theta<1-\alpha$ and define
$$
\tilde{t}_{N,T}(\alpha)=\mbox{argmax}_{\lfloor T\alpha \rfloor \leq t \leq T-\lfloor T\alpha \rfloor}
\sum_{i=1}^N\left(S_i(t)-\frac{t}{T}S_i(T)\right)^2.
$$

\medskip
\begin{lemma}\label{first} If Assumptions \ref{b-1}--\ref{cond-4} are satisfied, then we have
\beq\label{ff-10}
\lim_{\min(N,T)\to \infty}P\{\hat{t}_{N,T}=\tilde{t}_{N,T}(\alpha)\}=1\;\;\;\mbox{for all}\;\;\;0<\alpha<\theta<1-\alpha
\eeq
and
\beq\label{ff-2}
\frac{\hat{t}_{N,T}}{T}\;\;\stackrel{P}{\to}\;\;\theta\;\;\;\mbox{as}\;\;\min(T,N)\to \infty
\eeq
\end{lemma}

\begin{proof}
It is easy to see that for every $0<\alpha<\theta<1-\alpha$
$$
\frac{1}{T^2}\max_{\lfloor T\alpha \rfloor \leq t \leq T- \lfloor T\alpha \rfloor}r^2(t)\;\;\to \;\;\theta^2(1-\theta)^2,
$$
$$
\frac{1}{T^2}\max_{1\leq t\leq \lfloor T\alpha \rfloor }r^2(t)\;\;\to \;\;\alpha^2(1-\theta)^2,
$$
$$
\frac{1}{T^2}\max_{T-\lfloor T\alpha \rfloor \leq t \leq  T}r^2(t)\;\;\to \;\;\alpha^2\theta^2.
$$
We prove that
\beq\label{fp-0}
\sup_{1\leq t \leq T}\sum_{i=1}^N Q_i^2(t)=O_P(NT).
\eeq
Elementary arguments give
$$
EQ_i^2(t)\leq U_{i,2}(t)+ 2\left\{U_{i,2}(t)U_{i,2}(T)    \right\}^{1/2} +U_{i,2}(T)
$$
and therefore by Assumption \ref{cond-1}(i) we have
$$
\max_{1\leq t \leq T}\sum_{i=1}^NEQ_i^2(t)=O(NT).
$$
Let $q_i(u)=(Q_i^2(uT)-EQ_i^2(uT))/T$. Using Assumption \ref{b-1}(i), for every $0\leq u <v\leq 1$ by the Rosenthal inequality (cf.\ Petrov (1995,\ p.\ 59)) we have with $\nu=\kappa/2$
\begin{align*}
E\left| \sum_{i=1}^N(q_i(v)-q_i(u))\right|^{\nu}\leq c \left\{\sum_{i=1}^NE|q_i(v)-q_i(u)|^\nu +\left(\sum_{i=1}^NE(q_i(v)-q_i(u))^2
\right)^{\nu/2}
\right\}.
\end{align*}
By the Cauchy--Schwarz inequality we conclude for all $1\leq s\leq t\leq T$ that
\begin{align*}
E(Q_i^2(t)-Q_i^2(s))^2&\leq E\{(Q_i(t)-Q_i(s))^2(|Q_i(t)|+|Q_i(s)|)^2\}\\
&\leq 4(E(Q_i(t)-Q_i(s))^4)^{1/2}(EQ_i^4(t)+EQ_i^4(s))^{1/2}.
\end{align*}
Using the definition of $Q_i(t)$ and Assumption \ref{b-1}(ii) we get that
\begin{align*}
E(Q_i(t)-Q_i(s))^4\leq 2^3\left\{U_{i,4}(t-s)+\left(\frac{t-s}{T}\right)^4U_{i,4}(T)
\right\}.
\end{align*}
and similarly
$$
EQ_i^4(t)\leq 2^3\left( U_{i,4}(t)+U_{i,4}(T)  \right).
$$
Also,
\begin{align*}
(EQ_i^2(t)-EQ_i^2(s))^2&\leq 2E(Q_i(t)-Q_i(s))^2(EQ_i^2(t)+EQ_i^2(s))\\
&\leq 8\left(  U_{i,2}(t-s)+\left(\frac{t-s}{T}\right)^2U_{i,2}(T)   \right)\left(U_{i,2}(t)+U_{i,2}(s)+2U_{i,2}(T)
\right).
\end{align*}
Thus applying Assumption \ref{cond-1}(ii)  we get that with some $0<c<\infty$
$$
\left(\frac{1}{N}\sum_{i=1}^NE(q_i(v)-q_i(u))^2\right)^{\nu/2}\leq c|u-v|^{\nu/2}\;\;\mbox{for all}\;\;0\leq u<v\leq 1.
$$
Repeating the arguments used above we obtain that
\begin{align*}
E|Q_i^2(t)-Q_i^2(s)|^\nu&\leq E\{|Q_i(t)-Q_i(s)|^\nu(|Q_i(t)|+|Q_i(s)|)^\nu\}\\
&\leq 2^\nu(E|Q_i(t)-Q_i(s)|^{2\nu})^{1/2}(E|Q_i(t)|^{2\nu}+E|Q_i(s)|^{2\nu})^{1/2},
\end{align*}
\begin{align*}
E|Q_i(t)-Q_i(s)|^{2\nu}\leq 2^{2\nu}\left\{U_{i,2\nu}(t-s)+\left(\frac{t-s}{T}\right)^{2\nu}U_{i,2\nu}(T)
\right\},
\end{align*}
$$
E|Q_i(t)|^{2\nu}\leq 2^{2\nu}\left( U_{i,2\nu}(t)+U_{i,2\nu}(T)  \right)
$$
and
\begin{align*}
|EQ_i^2(t)-EQ_i^2(s)|^{2\nu}&\leq E|Q_i(t)-Q_i(s)|^\nu(E|Q_i(t)|^\nu+E|Q_i(s)|^\nu)\\
&\leq 2^{2\nu} \left(  U_{i,\nu}(t-s)+\left(\frac{t-s}{T}\right)^\nu U_{i,\nu}(T)   \right)\left(U_{i,\nu}(t)+U_{i,\nu}(s)+2U_{i,\nu}(T)
\right)
\end{align*}
resulting in
$$
\frac{1}{N}\sum_{i=1}^NE|q_i(u)-q_i(v)|^\nu\leq c|u-v|^{\nu/2}\;\;\;\mbox{for all  } 0\leq u,v\leq 1
$$
with some $c$. Using Billingsley (1968, pp.\ 95 and 127) we conclude that the process $\sum_{i=1}^Nq_i(u)/N$ is tight in
${\mathcal D}[0,1]$ and therefore \eqref{fp-0} holds.\\

The moment assumption in  Assumption \ref{con-eta} with the maximal inequality of M\'oritz et al.\ (1982) yields that $E(\max_{1\leq t \leq T}|V(t)|)^{\bar{\kappa}}=O(T^{\bar{\kappa}})$ and therefore by Markov's inequality we conclude
\beq\label{fp-2a}
\max_{1\leq t \leq T}|V(t)|=O_P(T^{1/2}).
\eeq
By \eqref{fp-2a} we get immediately that
\beq\label{fp-2}
\sup_{1\leq t \leq T}\sum_{i=1}^N\gamma_i^2V^2(t)=O_P(1)T\Gamma_{N,T}.
\eeq
Following the proof of \eqref{fp-0} we get
$$
\sup_{1\leq t \leq T}\left| \sum_{i=1}^N\gamma_iQ_i(t)  \right|=O_P(1)T^{1/2}\Gamma_{N,T}^{1/2}
$$
and therefore by \eqref{fp-2a}
\beq\label{fp-5}
\sup_{1\leq t \leq T}\left| \sum_{i=1}^N V(t)\gamma_iQ_i(t)   \right|=O_P(1) T\Gamma_{N,T}^{1/2} .
\eeq
Similarly to \eqref{fp-5} we have that
\beq\label{fp-6}
\sup_{1\leq t \leq T}\left| \sum_{i=1}^N r(t)\delta_iQ_i(t)   \right|=O_P(1) T^{3/2}\Delta_{N,T}^{1/2} .
\eeq
Using again Assumption \ref{con-eta}  and the definition of $r(t)$, one can easily verify that
\beq\label{fp-7}
\max_{1\leq t\leq T}\left|\sum_{i=1}^NV(t)r(t)\gamma_i\delta_i   \right|=O_P(1)T^{3/2}\left|\Sigma_{N,T}   \right|.
\eeq
It  follows from \eqref{fp-0}--\eqref{fp-7} that
$$
\frac{1}{T^2\Delta_{N,T}}\max_{1\leq t \leq T}\sum_{i=1}^N\left(S_i(t)-\frac{t}{T}S_i(T)\right)^2\;\;\stackrel{P}{\to}\;\;\theta^2(1-\theta)^2,
$$
$$
\frac{1}{T^2\Delta_{N,T}}\max_{1\leq t \leq \lfloor \alpha T\rfloor}\sum_{i=1}^N\left(S_i(t)-\frac{t}{T}S_i(T)\right)^2\;\;\stackrel{P}{\to}\;\;\alpha^2(1-\theta)^2
$$
and
$$
\frac{1}{T^2\Delta_{N,T}}\max_{T-\lfloor \alpha T\rfloor \leq t \leq T}\sum_{i=1}^N\left(S_i(t)-\frac{t}{T}S_i(T)\right)^2\;\;\stackrel{P}{\to}\;\;\alpha^2\theta^2,
$$
which immediately implies Lemma \ref{first}.
\end{proof}

\medskip
According to \eqref{ff-10}, it is enough to consider the asymptotic behavior of
\begin{align*}
\hat{t}_{N,T}&=
\mbox{argmax}_{\lf \alpha T\rf\leq t \leq \lf (1-\alpha)T\rf}\left\{U_N(t)-U_N(t_0)\right\},
\end{align*}
with any $0<\alpha<\theta<1-\alpha<1$, where $U_N(t)$ is defined in \eqref{u-def}.
It is easy to see that
\begin{align}\label{deco}
U_N(t)&-U_N(t_0)\\
&=\sum_{i=1}^N\biggl\{\delta_i^2(r^2(t)-r^2(t_0))+Q_i^2(t)-Q_i^2(t_0)+\gamma_i^2(V^2(t)-V^2(t_0)) \notag\\
&\hspace{1cm}+2\gamma_i(Q_i(t)V(t)
-Q_i(t_0)V(t_0))+
2\delta_i(r(t)Q_i(t)-r(t_0)Q_i(t_0))\notag\\
&\hspace{1cm}+2\gamma_i\delta_i(V(t)r(t)-V(t_0)r(t_0))\biggl\}.\notag
\end{align}

\begin{lemma}\label{tt-0} If Assumption \ref{t-0} holds, then for all $0<\alpha<\theta<1-\alpha$ there are $0<c_1, c_2<\infty$ such that
\begin{align*}
-c_1|t_0-t|T\leq r^2(t)-r^2(t_0)\leq -c_2|t_0-t|T
\end{align*}
for all  $1\leq t \leq T$. 
\end{lemma}

\begin{proof}
 The result follows from Assumption \ref{t-0} and the definition of $r(t)$.
\end{proof}

Throughout Lemmas \ref{qq-0}--\ref{ga-de} we assume that $1\leq M\leq T$.

\medskip

\begin{lemma}\label{qq-0} If Assumptions \ref{b-1}--\ref{cond-1} hold, then
\begin{align}\label{qq-1}
\max_{|t-t_0|\geq M}\frac{1}{|t-t_0|}\left|\sum_{i=1}^N[Q_i^2(t)-Q_i^2(t_0)]\right|=O_P(N+T^{2/\kappa}N^{1/2}+ (NT/M)^{1/2}).
\end{align}
\end{lemma}
\begin{proof} We write
$$
Q_i^2(t)-Q_i^2(t_0)=2Q_i(t_0)(Q_i(t)-Q_i(t_0))+(Q_i(t)-Q_i(t_0))^2.
$$
Using Assumption \ref{cond-1}(i) we get that
\beq\label{pq-1}
\sup_{1\leq t \leq T}\sum_{i=1}^N\frac{|EQ_i^2(t)-EQ_i^2(t_0)|}{|t-t_0|}=O(N).
\eeq
Let $\chi_i(t)=(Q_i(t)-Q_i(t_0))^2-E(Q_i(t)-Q_i(t_0))^2$.  Elementary arguments give
\begin{align}\label{ref-1}
P&\left\{\max_{1\leq t\leq T}\frac{1}{|t-t_0|}\left|\sum_{i=1}^N\chi_i(t)
\right|\geq x\right\}\\
&\hspace{.5 cm}=
P\left\{\left|\sum_{i=1}^N\chi_i(t)
\right|\geq x|t-t_0|\;\;\mbox{for at least one  } 1\leq t\leq T\right\}\notag\\
&\hspace{.5 cm}\leq
\sum_{1\leq t \leq T}P\left\{\left|\sum_{i=1}^N\chi_i(t)
\right|\geq x|t-t_0|
\right\}\notag\\
&\hspace{.5 cm}\leq
\sum_{1\leq t\leq T}
(x|t-t_0|)^{-\nu}E\left|\sum_{i=1}^N\chi_i(t)
\right|^\nu,\notag
\end{align}
where in the last step we used Markov's inequality. Let $\nu=\kappa/2$, where $\kappa$ is given in Assumption \ref{cond-1}(ii). The processes $\chi_i(t)$ are independent in $i$, so using the Rosenthal inequality (cf.\ Petrov (1995, p.\ 59)) we conclude with some $c>0$, not depending on $t$,
\begin{align}\label{ref-2}
E&\left|\sum_{i=1}^N\chi_i(t)
\right|^\nu
 \leq c\Biggl\{\sum_{i=1}^NE\left|\chi_i(t)
\right|^\nu
+\left(\sum_{i=1}^N E\chi_i^2(t)
\right)^{\nu /2}
\Biggl\}.\notag
\end{align}
Assumptions \ref{b-1}(i) and \ref{cond-1}(ii)  yield
\beq\label{ref-3}
\sum_{i=1}^NE\left|\chi_i(t)
\right|^\nu \leq c N|t-t_0|^\nu
\eeq
and
\beq\label{ref-4}
\left(\sum_{i=1}^NE\chi^2_i(t)
\right)^{\nu /2}\leq c N^{\nu/2}|t-t_0|^\nu.
\eeq
Thus we conclude via \eqref{ref-1}--\eqref{ref-4}
\begin{align*}
P&\left\{\max_{1\leq t \leq T}\frac{1}{|t-t_0|}\left|\sum_{i=1}^N\chi_i(t)
\right|\geq x\right\}\leq c
\sum_{1\leq t \leq T}
\frac{N^{\nu/2}|t-t_0|^\nu}{(x|t-t_0|)^{\nu}}
\leq \frac{c}{x^{\nu}}TN^{\nu/2}.
\end{align*}
Choosing $x=c_* N^{1/2}T^{1/\nu}$ with a large enough $c_*$ we get that
$$
\max_{1\leq t \leq T}\frac{1}{|t-t_0|}\left|\sum_{i=1}^N\chi_i(t)\right|=O_P(1)T^{1/\nu}N^{1/2}.
$$
Let
\beq\label{crlS}
\calS_i(t)=\sum_{s=1}^t e_{i,s}.
\eeq
It follows from the definition of $Q_i(t)$ that
\begin{align*}
&\left|\sum_{i=1}^N (Q_i(t_0)(Q_i(t)-Q_i(t_0))-E[Q_i(t_0)(Q_i(t)-Q_i(t_0))])\right|\\
&\hspace{1 cm}\leq \left|\sum_{i=1}^N( Q_i(t_0)(\calS_i(t_0)-\calS_i(t))-E[Q_i(t_0)(\calS_i(t_0)-\calS_i(t))])\right|\\
&\hspace{1.9cm}+\frac{|t-t_0|}{T}\left|\sum_{i=1}^N(Q_i(t_0)\calS_i(T)-E[Q_i(t_0)\calS_i(T)])
\right|.
\end{align*}
Using Assumption \ref{b-1}(i), and \ref{cond-1}(ii) with the Cauchy--Schwarz inequality we get that
\begin{align*}
\mbox{var}\left(\sum_{i=1}^NQ_i(t_0)\calS_i(T)\right)=\sum_{i=1}^N\mbox{var}\left(Q_i(t_0)\calS_i(T)\right)\leq \sum_{i=1}^N(EQ^4_i(t_0)E\calS^4_i(T))^{1/2}
=O(NT^2)
\end{align*}
and therefore
$$
\left|\sum_{i=1}^N(Q_i(t_0)\calS_i(T)-E[Q_i(t_0)\calS_i(T)])\right|=O_P(N^{1/2}T).
$$
With $\zeta_i(t)=Q_i(t_0)e_{i,t}-E[Q_i(t_0)e_{i,t}]$ we can write for $1\leq t \leq t_0$ that
\begin{align*}
\sum_{i=1}^N( Q_i(t_0)(\calS_i(t_0)-\calS_i(t))-E[Q_i(t_0)(\calS_i(t_0)-\calS_i(t))])=\sum_{s=t+1}^{t_0}\sum_{i=1}^N\zeta_i(s).
\end{align*}
By the Markov inequality we have
\begin{align}\label{mo-q}
P&\left\{\max_{1\leq t\leq t_0-M}\frac{1}{t_0-t}\left|\sum_{s=t+1}^{t_0}\sum_{i=1}^N\zeta_i(s)
\right|\geq x(NT/M)^{1/2}\right\}\\
&\hspace{.5 cm}\leq P\left\{\max_{\log M\leq k\leq \log t_0}\max_{e^k\leq \ell \leq e^{k+1}}\frac{1}{\ell}\left|\sum_{s=t_0-\ell}^{t_0}\sum_{i=1}^N\zeta_i(s)
\right|\geq x(NT/M)^{1/2}
\right\}  \notag\\
&\hspace{.5 cm}\leq P\biggl\{\max_{e^k\leq \ell \leq e^{k+1}}\frac{1}{\ell}\left|\sum_{s=t_0-\ell}^{t_0}\sum_{i=1}^N\zeta_i(s)
\right|\geq x(NT/M)^{1/2}\notag\\
&\hspace{3 cm}\mbox{for at least one    } \log M\leq k\leq \log t_0
\biggl\}  \notag\\
&\hspace{.5 cm}\leq\sum_{k=\log M}^{\log t_0}
P\left\{\max_{e^k\leq \ell \leq e^{k+1}}\left|\sum_{s=t_0-\ell}^{t_0}\sum_{i=1}^N\zeta_i(s)
\right|\geq x(NT/M)^{1/2}e^{k}
\right\}\notag\\
&\hspace{.5 cm}\leq(x(NT/M)^{1/2})^{-\nu}\sum_{k=\log M}^{\log t_0}e^{-k\nu}E\max_{e^k\leq \ell \leq e^{k+1}}\left|\sum_{s=t_0-\ell}^{t_0}\sum_{i=1}^N\zeta_i(s)
\right|^\nu.\notag
\end{align}
Next we need a maximal inequality for  double sum in the last term above.  With $\bar{\zeta}_i(s)= \zeta_i(t_0-s+1)  $ we get that
$$
\sum_{s=t_0-\ell+1}^{t_0}\sum_{i=1}^N\zeta_i(s)=\sum_{s=1}^\ell\sum_{i=1}^N\bar{\zeta}_i(s).
$$
By the independence of the processes $\bar{\zeta}_i(s)$ in  $i$,  Rosenthal's inequality (cf.\ Petrov (1995,\ p.59)) implies  that
\begin{align}\label{rosie}
E\left|\sum_{i=1}^N\sum_{s=u}^v\bar{\zeta}_i(s)\right|^\nu\leq c\left\{ \sum_{i=1}^NE\left| \sum_{s=u}^v\bar{\zeta}_i(s)\right|^\nu
+\left(  \sum_{i=1}^NE\left( \sum_{s=u}^v\bar{\zeta}_i(s)\right)^2 \right)^{\nu/2}\right\}.
\end{align}
We have via the Cauchy--Schwarz inequality
$$
E\left( \sum_{s=u}^v\bar{\zeta}_i(s)\right)^2\leq E\left(Q_i(t_0)\sum_{s=u}^ve_{i,s}\right)^2\leq \left(EQ_i^4(t_0)\right)^{1/2}\biggl( E\biggl(\sum_{s=u}^ve_{i,s}
\biggl)^4\biggl)^{1/2}
$$
and therefore
\begin{align*}
\sum_{i=1}^NE\left( \sum_{s=u}^v\bar{\zeta}_i(s)\right)^2&\leq \sum_{i=1}^N\left(EQ_i^4(t_0)\right)^{1/2}\biggl( E\biggl(\sum_{s=u}^ve_{i,s}
\biggl)^4\biggl)^{1/2}\\
&\leq \left\{\sum_{i=1}^NEQ_i^4(t_0)\sum_{i=1}^NE\left(\sum_{s=u}^ve_{i,s}
\right)^4\right\}^{1/2}.
\end{align*}
Using that the $e_{i,t}$'s have mean zero and Assumption \ref{cond-1}, we conclude
$$
\sum_{i=1}^NE\left( \sum_{s=u}^v\bar{\zeta}_i(s)\right)^2\leq c NT|u-v|.
$$
Similarly, by the definition of $\bar{\zeta}_i$ we get
\begin{align*}
E\left| \sum_{s=u}^v\bar{\zeta}_i(s)\right|^\nu\leq 2^\nu\left\{E\left|Q_i(t_0)\sum_{s=u}^ve_{i,s}\right|^\nu+
\left|E\left[Q_i(t_0)\sum_{s=u}^ve_{i,s}\right]\right|^\nu
\right\},
\end{align*}
and by applications of the Cauchy--Schwarz inequality we have
$$
\left|E\left[Q_i(t_0)\sum_{s=u}^ve_{i,s}\right]\right|\leq \left\{EQ_i^2(t_0)E\left[\sum_{s=u}^ve_{i,s}\right]^2\right\}^{1/2},
$$
$$
E\left|Q_i(t_0)\sum_{s=u}^ve_{i,s}\right|^\nu\leq \left\{E|Q_i(t_0)|^{2\nu}E\left|\sum_{s=u}^ve_{i,s}\right|^{2\nu}\right\}^{1/2}
$$
resulting in
\begin{align*}
\sum_{i=1}^NE\left| \sum_{s=u}^v\bar{\zeta}_i(s)\right|^\nu\leq c N T^{\nu/2}|u-v|^{\nu/2}.
\end{align*}
Using the inequalities above, we get the upper bound for the moment in \eqref{rosie}:
\begin{align}\label{rosie-2}
E\left|\sum_{s=u}^v\sum_{i=1}^N\bar{\zeta}_i(s)\right|^\nu\leq cN^{\nu/2}T^{\nu/2}|u-v|^{\nu/2}.
\end{align}
Applying the maximal inequality in M\'oritz et al.\ (1982) to \eqref{rosie-2} we conclude
$$
E\max_{e^{k}\leq \ell\leq e^{k+1}}\left|\sum_{s=1}^\ell\sum_{i=1}^N\bar{\zeta}_i(s)\right|^\nu\leq c N^{\nu/2}T^{\nu/2}e^{k\nu/2}.
$$
Hence \eqref{mo-q} implies that
\begin{align*}
P\left\{\max_{1\leq t\leq t_0-M}\frac{1}{t_0-t}\left|\sum_{s=t+1}^{t_0}\sum_{i=1}^N\zeta_i(s)
\right|\geq x(NT/M)^{1/2}\right\}&\leq \frac{c}{x^{\nu}(NT/M)^{\nu/2}}\sum_{k=\log M}^{\infty}e^{-k\nu/2}N^{\nu/2}T^{\nu/2}\\
&\leq \frac{c}{x^\nu},
\end{align*}
resulting in
$$
\max_{1\leq t\leq t_0-M}\frac{1}{t_0-t}\left|\sum_{s=t+1}^{t_0}\sum_{i=1}^N\zeta_i(s)
\right|=O_P(1) (NT/M)^{1/2}.
$$
Similar arguments yield
$$
\max_{t_0+M\leq t\leq T}\frac{1}{t_0-t}\left|\sum_{s=t+1}^{t_0}\sum_{i=1}^N\zeta_i(s)
\right|=O_P(1) (NT/M)^{1/2},
$$
which completes the proof of the lemma.
\end{proof}

\medskip

\begin{lemma}\label{r-q} If Assumptions \ref{b-1}--\ref{cond-1} hold, then
$$
\max_{|t-t_0|\geq M} \frac{1}{|t-t_0|} \left|\sum_{i=1}^N\delta_i(r(t)Q_i(t)-r(t_0)Q_i(t_0))\right|=O_P(1)(T^{1/2}\Delta_{N,T}^{1/2}+ T(\Delta_{N,T}/M)^{1/2}).
$$
\end{lemma}
\begin{proof} First we write
$$
\delta_i(r(t)Q_i(t)-r(t_0)Q_i(t_0))=r(t)\delta_i(Q_i(t)-Q_i(t_0))+\delta_iQ_i(t_0)(r(t)-r(t_0)).
$$
Applying the definition of $r(t)$ with Assumptions \ref{b-1}(i) and \ref{cond-1}(i), we get
\begin{align*}
\max_{|t-t_0|\leq T}\frac{|r(t)-r(t_0)|}{|t-t_0|}\left|\sum_{i=1}^N\delta_iQ_i(t_0)\right|=O_P(1)T^{1/2}\Delta_{N,T}^{1/2}.
\end{align*}
It follows from the definition of $Q_i(t)$ that for all $i$
$$
Q_i(t_0)-Q_i(t)=Z_{i}(t,t_0)
-\frac{t_0-t}{T}\sum_{s=1}^Te_{i,s},\;\;\mbox{if}\;\;1\leq t \leq T,
$$
where
\begin{displaymath}
Z_i(t,t_0)=\left\{
\begin{array}{ll}
&\displaystyle \sum_{s=t+1}^{t_0}e_{i,s},\;\;\mbox{if}\;\;1\leq t < t_0\\
&0,\;\;\mbox{if}\;\;t=t_0\\
&\displaystyle-\sum_{s=t_0+1}^{t}e_{i,s},\;\;\mbox{if}\;\;t_0< t \leq T.
\end{array}
\right.
\end{displaymath}
Clearly, $(\sum_{i=1}^N\delta_i\sum_{s=1}^Te_{i,s})^2=O(T\Delta^{1/2}_{N,T})$ on account of Assumptions \ref{b-1}(i) and \ref{cond-1}(i) and therefore
\begin{align*}
\max_{|t-t_0|\leq T}\frac{|r(t)|}{|t-t_0|}\left|\frac{t-t_0}{T}\sum_{i=1}^N\delta_i\sum_{s=1}^Te_{i,s}
\right|\leq \left|\sum_{i=1}^N\delta_i\sum_{s=1}^Te_{i,s}
\right|=O_P(1)T^{1/2}\Delta_{N,T}^{1/2}.
\end{align*}
Repeating the arguments used in \eqref{mo-q},
by Markov's inequality we have
\begin{align}\label{sto-0}
P&\left\{\max_{1 \leq t \leq t_0-M}\frac{1}{t_0-t}\left|\sum_{i=1}^N\delta_iZ_i(t,t_0)\right|\geq x (\Delta_{N,T}/M)^{1/2}\right\}\\
&\leq P\left\{\max_{\log M \leq j \leq \log T}\max_{e^j\leq \ell \leq e^{j+1}}\frac{1}{\ell}\left|\sum_{i=1}^N\delta_iZ_i(t_0-\ell,t_0)\right|\geq x (\Delta_{N,T}/M)^{1/2}\right\}  \notag\\
&\leq \sum_{j= \log M}^\infty P\left\{\max_{e^j\leq \ell \leq e^{j+1}}\left|\sum_{i=1}^N\delta_iZ_i(t_0-\ell,t_0)\right|\geq x e^j (\Delta_{N,T}/M)^{1/2}\right\}\notag\\
&\leq \frac{(M/\Delta_{N,T})^{\nu/2}}{x^\nu }\sum_{j= \log M}^\infty
e^{-j\nu}E\max_{e^j\leq \ell \leq e^{j+1}}\left|\sum_{i=1}^N\delta_iZ_i(t_0-\ell,t_0)\right|^\nu.\notag
\end{align}
With $\bar{e}_{i,s}=e_{i,t_0-s+1}$ we
\begin{align*}
\max_{e^j\leq \ell \leq e^{j+1}}\left|\sum_{i=1}^N\delta_iZ_i(t_0-\ell,t_0)\right|
=\max_{e^j\leq \ell \leq e^{j+1}}\left|
\sum_{s=1}^{\ell+1}\sum_{i=1}^N\delta_i\bar{e}_{i,s}
\right|.
\end{align*}
Using Assumptions \ref{b-1}(i), \ref{cond-1}(ii) and \ref{cond-4}(ii) with Rosenthal's inequality we conclude for all $\nu>2$ that
\begin{align}\label{sto-1}
E\left|\sum_{s=u}^{v}\sum_{i=1}^N\delta_i\bar{e}_{i,s}\right|^\nu\leq c  (v-u)^{\nu/2}\Biggl\{
\sum_{i=1}^N |\delta_i|^{\nu}+\Delta_{N,T}^{\nu/2}\Biggl\}\leq 2 c(v-u)^{\nu/2}\Delta_{N,T}^{\nu/2},
\end{align}
since by the multinomial theorem
$$
\sum_{i=1}^N |\delta_i|^{\nu}\leq \Delta_{N,T}^{\nu/2}.
$$
The maximal inequality  of M\'oricz et al.\ (1982) and  \eqref{sto-1} imply that
$$
E\max_{e^j\leq \ell \leq e^{j+1}}\left|\sum_{i=1}^N\delta_iZ_i(t_0-\ell,t_0)\right|^\nu\leq c \Delta_{N,T}^{\nu/2}e^{j\nu/2}.
$$
and therefore by \eqref{sto-1} we have
\begin{align}\label{sto-3}
P&\left\{\max_{1 \leq t \leq t_0-M}\frac{1}{t_0-t}\left|\sum_{i=1}^N\delta_iZ_i(t,t_0)\right|\geq x (\Delta_{N,T}/M)^{1/2}\right\}\\
&\leq  c \frac{(M/\Delta_{N,T})^{\nu/2}}{x^\nu }\sum_{j= \log M}^\infty e^{-j\nu/2}\Delta_{N,T}^{\nu/2}
\notag \\
&\leq \frac{c}{x^\nu}.  \notag
\end{align}
Thus we conclude that
$$
\max_{1\leq t \leq t_0-M}\frac{1}{t_0-t}\left|\sum_{i=1}^N\delta_iZ_i(t,t_0)\right|=O_P(1)(\Delta_{N,T}/M)^{1/2}
$$
and by similar arguments we have
$$
\max_{t_0+M\leq  t \leq T}\frac{1}{t-t_0}\left|\sum_{i=1}^N\delta_iZ_i(t,t_0)\right|=O_P(1)(\Delta_{N,T}/M)^{1/2},
$$
which also completes the proof of the lemma.
\end{proof}

\medskip
\begin{lemma}\label{gV} If Assumptions \ref{b-1}--\ref{con-eta} hold, then
$$
\max_{|t-t_0|\geq M} \frac{1}{|t-t_0|}\left|\sum_{i=1}^N \gamma_i^2(V^2(t)-V^2(t_0))\right|=O_P(1)\Gamma_{N,T}\left(1 +(\log (T/ M))^{2/\bark}+M^{-1/2}T^{1/2}\right).
$$
\end{lemma}

\begin{proof} We write $|V^2(t)-V^2(t_0)|\leq (V(t)-V(t_0))^2+2|V(t_0)| |V(t)-V(t_0)| $. If $1\leq t \leq t_0$, then
$$
|V(t)-V(t_0)|\leq \left|\sum_{s=t+1}^{t_0} \eta_s\right|+\frac{|t-t_0|}{T}\left|\sum_{s=1}^T\eta_s\right|
$$
and therefore
$$
(V(t)-V(t_0))^2\leq 4\left(\sum_{s=t+1}^{t_0} \eta_s\right)^2+4\frac{(t-t_0)^2}{T^2}\left(\sum_{s=1}^T\eta_s\right)^2.
$$
Thus we get from Assumption \ref{con-eta} that
$$
\max_{|t-t_0|\geq M} \frac{(V(t)-V(t_0))^2}{t_0-t}=O_P(1)+2\left(\max_{|t-t_0|\geq M}\frac{1}{(t_0-t)^{1/2}}\left|\sum_{s=t+1}^{t_0} \eta_s\right|\right)^2.
$$
Repeating the arguments used in \eqref{mo-q}  we get that
\begin{align*}
P&\left\{\max_{t_0-t\geq M} \frac{1}{(t_0-t)^{1/2}}\left|\sum_{s=t+1}^{t_0} \eta_s\right|\geq x
\right\}\\
&\leq P\left\{\max_{\log M\leq k\leq \log T}\max_{e^k\leq u\leq e^{k+1}} \frac{1}{u^{1/2}}\left|\sum_{s=t_0-u+1}^{t_0} \eta_s\right|\geq x
\right\}\\
&\leq \sum_{k=\log M}^{\log T}P\left\{\max_{e^k\leq u\leq e^{k+1}}
\left|\sum_{s=t_0-u+1}^{t_0} \eta_s\right|\geq x e^{k/2}
\right\}\\
&\leq \frac{1}{x^{\bark}}\sum_{k=\log M}^{\log T}e^{-k\bark /2}E\max_{e^k\leq u\leq e^{k+1}}\left|\sum_{s=t_0-u+1}^{t_0} \eta_s\right|^{\bark}.
\end{align*}
Following the arguments used in the proofs of Lemmas \ref{qq-0} and \ref{r-q} one can verify that
$$
E\max_{e^k\leq u\leq e^{k+1}}\left|\sum_{s=t_0-u+1}^{t_0} \eta_s\right|^{\bark}\leq c e^{k\bark /2}
$$
which implies that
\begin{align*}
P\left\{\max_{t_0-t\geq M} \frac{1}{(t_0-t)^{1/2}}\left|\sum_{s=t+1}^{t_0} \eta_s\right|\geq x
\right\}\leq c \frac{\log (T/ M)}{x^{\bark}}.
\end{align*}
Similar computations can be performed for $t-t_0\geq M$ and thus we conclude
$$
\max_{|t_0-t|\geq M} \frac{1}{(t_0-t)^{1/2}}\left|\sum_{s=t+1}^{t_0} \eta_s\right|=O_P((\log (T/ M))^{1/{\bark}}).
$$
As in  the proof of Lemma \ref{qq-0}  we have that
\begin{align}\label{gd-3}
\sup_{t_0-t\geq M}\frac{1}{t_0-t}\left|\sum_{s=t+1}^{t_0}\eta_s\right|=O_P(M^{-1/2})\;\;\;\mbox{and}
\;\;\;\sup_{t-t_0\geq M}\frac{1}{t-t_0}\left|\sum_{s=t_0+1}^{t}\eta_s\right|=O_P(M^{-1/2}).
\end{align}
The proof of the lemma is now complete.
\end{proof}
\medskip

\begin{lemma}\label{ga-q} If Assumptions \ref{b-1}--\ref{con-eta} hold, then
$$
\max_{|t-t_0|\geq M}\frac{1}{|t-t_0|}\left|\sum_{i=1}^N\gamma_i(Q_i(t)V(t)-Q_i(t_0)V(t_0))    \right|
=O_P(1)\Gamma_{N,T}^{1/2}T^{1/2}M^{-1/2}.
$$
\end{lemma}
\begin{proof} We write
$$
Q_i(t)V(t)-Q_i(t_0)V(t_0)=V(t)(Q_i(t)-Q_i(t_0))+(V(t)-V(t_0))Q_i(t_0).
$$
Assumption \ref{con-eta} implies that
$$
\max_{1\leq t \leq T}|V(t)|=O_P(T^{1/2})
$$
and by the arguments used  in the proof of Lemma \ref{r-q} one can show that
$$
\max_{|t-t_0|\geq M}\frac{1}{|t-t_0|}\left|\sum_{i=1}^N\gamma_i(Q_i(t)-Q_i(t_0))\right|=O_P(1)
M^{-1/2}\Gamma_{N,T}^{1/2}.
$$
Similar arguments yield
$$
\max_{|t-t_0|\geq M}\frac{|V(t)-V(t_0)|}{|t-t_0|}\left|\sum_{i=1}^N\gamma_iQ_i(t_0)\right|=O_P(1)
T^{1/2}M^{-1/2}\Gamma_{N,T}^{1/2}.
$$
\end{proof}

\medskip

\begin{lemma}\label{ga-de} If Assumptions \ref{b-1}--\ref{con-eta} hold, then
$$
\max_{|t-t_0|\geq M}\frac{1}{|t-t_0|}\left|\sum_{i=1}^N\gamma_i\delta_i(V(t)r(t)-V(t_0)r(t_0))    \right|=O_P(1)(T^{1/2}+TM^{-1/2})\left|
\Sigma_{N,T}\right|.
$$
\end{lemma}
\begin{proof} Since $V(t)r(t)-V(t_0)r(t_0)=V(t_0)(r(t)-r(t_0))+r(t)(V(t)-V(t_0))$,
Lemma \ref{ga-de} follows from
\beq\label{gd-1}
\max_{|t-t_0|\geq M}\left|V(t_0)\frac{r(t)-r(t_0)}{t-t_0}\right|=O_P(1)T^{1/2}
\eeq
and
\beq\label{gd-2}
\max_{|t-t_0|\geq M}\left|r(t)\frac{V(t)-V(t_0)}{t-t_0}\right|=O_P(1)TM^{-1/2}.
\eeq
The claim in \eqref{gd-1} is an immediate consequence of the definition of $r(t)$ and Assumption \ref{con-eta} while \eqref{gd-2} is proven in \eqref{gd-3}.
\end{proof}

\medskip
\noindent
{\it Proof of Theorem \ref{rate-con-1}}
Under assumptions \eqref{m-cho-2} and \eqref{n-1} we use Lemmas \ref{tt-0}--\ref{ga-de} with $M=1$.
\qed

 \medskip
\noindent
{\it Proof of Remark \ref{rem-bai}.} The proof of this remark follows Bai (2010) closely. We use \eqref{summa}. Since $T$ is fixed,
$$
\max_{1\leq t \leq T}\sum_{i=1}^N\gamma_i^2V^2(t)=O_P(1)\Gamma
$$
and by Assumption \ref{cond-1}(i) and Markov's inequality we have
$$
\max_{1\leq t \leq T}\sum_{i=1}^NQ_i^2(t)=O_P(N).
$$
By the Cauchy--Schwarz inequality and Assumption \ref{cond-1}(i) we conclude
$$
E\left|\sum_{i=1}^NV(t)\gamma_iQ_i(t)\right|=O(1)\Gamma^{1/2}
$$
and therefore
$$
\max_{1\leq t \leq T}\left|\sum_{i=1}^NV(t)\gamma_iQ_i(t)\right|=O_P(1)\Gamma^{1/2}.
$$
Similar arguments give
$$
\max_{1\leq t \leq T}\left|r(t)\sum_{i=1}^N\delta_iQ_i(t)\right|=O_P(\Delta^{1/2})
$$
and
$$
\max_{1\leq t \leq T}\left|r(t)V(t)\sum_{i=1}^N\delta_i\gamma_i\right|=O_P(|\Sigma|).
$$
The final term coming from \eqref{summa} to consider is $\Sigma_{i=1}^N r^2(t)\delta_i^2=\Delta r^2(t)$. Under the conditions of the remark, this is the asymptotically dominating term which has a unique maximum at $t_0$. Hence Remark \ref{rem-bai} is proven.
\qed\\
\medskip
\noindent
{\it Proof of Remark \ref{rem-bai-cor}.} Let
$$
f_i(t)=\frac{1}{(t(T-t))^{1/2}}\left(Q_i(t)+\gamma_iV(t)\right),
$$
where $Q_i(t)$ and $V(t)$ are defined in \eqref{q-v-def}. We note that due to the assumption that the $e_{i,s}$ and $\eta_t$ are sequences of uncorrelated random variables  we get that
\beq\label{husk-cor-1}
Ef_i(t)^2=\sigma_i^2+\gamma_i^2
\eeq
and
\beq\label{husk-cor-2}
\mbox{var}(f_i^2(t))\leq C_1(Ee_{i,0}^4+\gamma_i^4)
\eeq
with some constant $C_1$.
We write
$$
\sum_{i=1}^N\left(S_i(t)-\frac{t}{T}S_i(t)\right)^2\frac{1}{t(T-T)}=\cH_{1,N}(t)+\cH_{2,N}(t)+\cH_{3,N}(t),
$$
with
$$
\cH_{1,N}(t)=\Delta \frac{r^2(t)}{t(T-T)},\;\;\;\cH_{2,N}(t)=\sum_{i=1}^N f_{i}^2(t)\;\;\;\mbox{and}\;\;\;\cH_{3,n}(t)=2\sum_{i=1}^Nf_{i}(t)\frac{\delta_ir(t)}{(t(T-t))^{1/2}},
$$
where  $r(t), t=1,2,\ldots ,T$ is defined in \eqref{r-def}. We show that for all $t\neq t_0$
\beq\label{husk-cor-3}
\lim_{N\to \infty}P\{\cH_{1,N}(t_0)-\cH_{1,N}(t)\leq \cH_{2,N}(t)-\cH_{2,N}(t_0)+\cH_{2,N}(t)-\cH_{2,N}(t_0)\}=0,
\eeq
which immediately implies Remark \ref{rem-bai-cor}. We note that with some $C_2>0$ we have that $\cH_{1,N}(t_0)-\cH_{1,N}(t)\geq C_2\Delta$ for all $t\neq t_0$.
By the independence of the processes $Q_i(t), 1\leq i \leq N$ and V(t)  we conclude
\begin{align*}
E\left(\cH_{2,N}(t)-\cH_{2,N}(t_0)\right)^2&=\sum_{i=1}^NE\left(\frac{1}{t(T-t)}Q^2_i(t)-\frac{1}{t_0(T-t_0)}Q^2_i(t_0)\right)^2\\
&\hspace{.5cm}+4\sum_{i=1}^NE\left(\frac{1}{t(T-t)}Q_i(t)\gamma_iV(t)-\frac{1}{t_0(T-t_0)}Q_i(t_0)\gamma_iV(t)\right)^2\\
&\hspace{.5cm}+E\left(\frac{V^2(t)}{t(T-t)}-\frac{V^2(t_0)}{t_0(T-t_0)}\right)^2\Gamma^2\\
&=O(N+\Gamma+\Gamma^2)
\end{align*}
and therefore
$$
\max_{1\leq t <T}\left|\cH_{2,N}(t)-\cH_{2,N}(t_0))\right|=O_P(1)\left(N^{1/2}+\Gamma^{1/2}+\Gamma\right).
$$
Similarly,
$$
\max_{1\leq t <T}\left|\cH_{3,N}(t)-\cH_{3,N}(t_0))\right|=O_P(1)\left( \Delta^{1/2}+|\Sigma| \right)=O_P(1)\left( \Delta^{1/2}+\Delta^{1/2}\Gamma^{1/2} \right),
$$
since $|\Sigma|\leq \Delta^{1/2}\Gamma^{1/2}$, completing the proof of \eqref{husk-cor-3}.
\qed

\medskip
\begin{lemma}\label{rate-con} We assume that Assumptions \ref{b-1}--\ref{m-cho-1} hold, and $|{\frak s}|<\infty.$
Then, as $N,T\to \infty$ we have that
\beq\label{r-co-1}
\Delta_{N,T} |\hat{t}_{N,T}-t_0|=O_P(1).
\eeq
\end{lemma}

\begin{proof} By Lemma \ref{first} it is enough to prove that for all $0<\alpha<\theta$
\beq\label{r-co-2}
\Delta |\tilde{t}_{N,T}(\alpha)-t_0|=O_P(1)
\eeq
Under Assumption \ref{m-cho-1}(i) we choose $M=C/\Delta_{N,T}$, where $C>0$ is a constant.
Using Lemmas \ref{tt-0}, \ref{qq-0}, \ref{gV}--\ref{ga-de} and Assumption \ref{cond-4}(ii) we obtain that
\begin{align}\label{deco-1}
&\Delta_{N,T}(r^2(t)-r^2(t_0))+\sum_{i=1}^N(Q_i^2(t)-Q_i^2(t_0))+\Gamma_{N,T}(V^2(t)-V^2(t_0))\\
&\hspace{1cm}+2\sum_{i=1}^N\gamma_i(Q_i(t)V(t)
-Q_i(t_0)V(t_0))
+2\sum_{i=1}^N\gamma_i\delta_i(V(t)r(t)-V(t_0)r(t_0))\notag\\
&=\Delta_{N,T}(r^2(t)-r^2(t_0))(1+o_P(1))\;\;\mbox{uniformly on}\;\;|t-t_0|\geq M\notag
\end{align}
for all $M$. Also, by Lemmas \ref{tt-0} and \ref{r-q} we obtain that
\begin{align*}
\lim_{C\to\infty}\liminf_{N,T\to\infty}P\left\{\sup_{|t-t_0|\geq C/\Delta_{N,T}}\sum_{i=1}^N\left[\delta_i^2(r^2(t)-r^2(t_0))+2\delta_i(r(t)Q_i(t)-r(t_0)Q_i(t_0))
\right]<0
\right\}=1.
\end{align*}
Hence Lemma \ref{rate-con} is established under Assumption \ref{m-cho-1}(i) and $|{\frak s}|<\infty$.
\end{proof}

\medskip
\begin{lemma}\label{lem-w}  We assume that Assumptions \ref{b-1}--\ref{m-cho-1} hold, and $|{\frak s}|<\infty.$
Then, as $N,T\to \infty$ we have that
\beq\label{w-1-*}
\sup_{|t-t_0|\leq C/\Delta}\left|\frac{1}{T}\sum_{i=1}^N\delta_i^2(r^2(t_0)-r^2(t))-2\theta(1-\theta)\Delta g_\theta(t-t_0)
\right|=o(1),
\eeq
\begin{align}\label{w-2-*}
\sup_{|t-t_0|\leq C/\Delta}&\biggl|\frac{1}{T}\sum_{i=1}^N\delta_i(Q_i(t)r(t)-Q_i(t_0)r(t_0))\\
&\hspace{4 cm}+\theta(1-\theta)\sum_{i=1}^N\delta_i(\calS_i(t)-\calS_i(t_0))\notag
\biggl|\\
&=o_P(1),\notag
\end{align}
\beq\label{w-6}
\sup_{|t-t_0|\leq C/\Delta}\left|\frac{1}{T}\sum_{i=1}^N \gamma_i\delta_i(V(t)r(t)-V(t_0)r(t_0))+\theta(1-\theta)\Sigma_{N,T}(V(t)-V(t_0))\right|=o_P(1),
\eeq
\beq\label{w-3}
\sup_{|t-t_0|\leq C/\Delta}\left|\frac{1}{T}\sum_{i=1}^N(Q_i^2(t)-Q_i^2(t_0))\right|=o_P(1),
\eeq
\beq\label{w-4}
\sup_{|t-t_0|\leq C/\Delta}\left|\frac{1}{T}\sum_{i=1}^N\gamma_i^2(V^2(t)-V^2(t_0))\right|=o_P(1)
\eeq
and
\beq\label{w-5}
\sup_{|t-t_0|\leq C/\Delta}\left|\frac{1}{T}\sum_{i=1}^N \gamma_i(Q_i(t)V(t)-Q_i(t_0)V(t_0))\right|=o_P(1),
\eeq

for all $C>0$, where $\Delta=\Delta_{N,T}$ and $\calS_i(\cdot)$ is defined in \eqref{crlS}.
\end{lemma}
\begin{proof} First we note
\bals
\frac{1}{T}\sum_{i=1}^N\delta_i^2(r^2(t)-r^2(t_0))=\frac{2r(t_0)}{T}&\sum_{i=1}^N\delta_i^2(r(t)-r(t_0))\\
&+\frac{1}{T}\sum_{i=1}^N\delta_i^2(r(t)-r(t_0))^2.
\end{align*}
Using the definition of $r(t)$ and Assumption \ref{cond-4}(i) we conclude
$$
\sup_{|t-t_0|\leq C/\Delta}\left|\frac{1}{T}\sum_{i=1}^N\delta_i^2(r(t)-r(t_0))^2\right|=O(1/(T\Delta))=o(1)
$$
and
$$
\sup_{|t-t_0|\leq C/\Delta}\left|\frac{2r(t_0)}{T}\sum_{i=1}^N\delta_i^2(r(t)-r(t_0))-2\theta(1-\theta)\Delta g_\theta(t-t_0)
\right|=o(1),
$$
completing the proof of \eqref{w-1-*}.\\
Similarly,
\bals
\sum_{i=1}^N\delta_i(Q_i(t)r(t)-Q_i(t_0)r(t_0))=r(t_0)\sum_{i=1}^N\delta_i(Q_i(t)-Q_i(t_0))+\sum_{i=1}^N\delta_iQ_i(t)(r(t)-r(t_0))
\end{align*}
and
\bals
\sum_{i=1}^N\delta_i(Q_i(t)-Q_i(t_0))=\sum_{i=1}^N\delta_i(\calS_i(t)-\calS_i(t_0))+\frac{t-t_0}{T}\sum_{i=1}^N\delta_i\calS_i(T).
\end{align*}
Computing the variance of $\sum_{i=1}^N\delta_i\calS_i(T)$  we get
$$
\sup_{|t-t_0|\leq C/\Delta}\left|\frac{r(t_0)}{T}\frac{t-t_0}{T}\sum_{i=1}^N\delta_i\calS_i(T)
\right|=O_P(1/(T\Delta)^{1/2})=o_P(1)
$$
by Assumption \ref{cond-4}(i), so \eqref{w-2-*} is proven.\\
Clearly,
\bals
V(t)r(t)-V(t_0)r(t_0)=(V(t)-V(t_0))r(t)+V(t_0)(r(t)-r(t_0)).
\end{align*}
By Assumption \ref{con-eta} we get that $V(t_0)=O_P(T^{1/2})$ and therefore
$$
\sup_{|t-t_0|\leq C/\Delta}|V(t_0)(r(t)-r(t_0))|=O_P(T^{1/2}/\Delta).
$$
We note that  for all $t_0 \leq t\leq t_0+C/\Delta$
$$
|V(t)-V(t_0)|\leq \left|\sum_{s=t_0+1}^{t}\eta_s\right| +\frac{|t-t_0|}{T}\left|\sum_{s=1}^{T}\eta_s\right|
$$
and by Assumption \ref{con-eta} we have that $|\sum_{s=t_0+1}^{t}\eta_s|=O_P(1/\Delta^{1/2})$ and the process $\Delta^{1/2}\sum_{s=1}^{\lf u/\Delta\rf}\eta_s, 0\leq u\leq 1$ is tight in ${\mathcal D}[0,C]$. Thus by stationarity we get
$$
\sup_{t_0\leq t \leq C/\Delta}\left|\sum_{s=t_0+1}^{t}\eta_s\right|=O_P(1/\Delta^{1/2})
$$
and similar arguments can be used on $t_0-C/\Delta\leq t \leq t_0$. We conclude that
\beq\label{V-comp}
\sup_{|t-t_0|\leq C/\Delta}|V(t)-V(t_0)|=O_P(1/\Delta^{1/2}+T^{-1/2}/\Delta)=O_P(1/\Delta^{1/2})
\eeq
completing the proof of \eqref{w-6} on account $|{\frak s}|<\infty$.\\
With $\psi_i(t)=Q_i^2(t)-Q_i^2(t_0)-E(Q_i^2(t)-Q_i^2(t_0))$  we can write
\bals
\sup_{|t-t_0|\leq C/\Delta}&\frac{1}{T}\left|\sum_{i=1}^N(Q^2_i(t)-Q_i^2(t_0))\right|\\
&\leq \sup_{|t-t_0|\leq C/\Delta}\frac{1}{T}\left|\sum_{i=1}^N(EQ^2_i(t)-EQ_i^2(t_0))\right|
+\sup_{|t-t_0|\leq C/\Delta}\frac{1}{T}\left|\sum_{i=1}^N\psi_i(t)\right|.
\end{align*}
We obtain from the proof of Lemma \ref{first} that
$$
\sup_{|t-t_0|\leq C/\Delta}\frac{1}{T}\left|\sum_{i=1}^N(EQ^2_i(t)-EQ_i^2(t_0))\right|=O\left(\frac{N}{T\Delta}\right)=o(1).
$$
For every $t\in[t_0-C/\Delta, t_0+C/\Delta]$
we have that
$$
E\left(\frac{1}{T}\sum_{i=1}^N\psi_i(t)\right)^2\leq C\frac{1}{T^2}\sum_{i=1}^NE\psi_i^2(t)
$$
and
\bals
\sup_{|t-t_0|\leq C/\Delta}E\psi_i^2(t)&\leq 9\sup_{|t-t_0|\leq C/\Delta}\{ Q_i^2(t)(Q_i(t)-Q_i(t_0))^2
+Q_i^2(t_0)(Q_i(t)-Q_i(t_0))^2
\}\\
&=O(1)\left\{\sup_{|t-t_0|\leq C/\Delta} (EQ_i^4(t))^{1/2}\sup_{|t-t_0|\leq C/\Delta}(E(Q_i(t)-Q_i(t_0))^4)^{1/2}
\right\}\\
&=O(T/\Delta).
\end{align*}
Next we show that $\sqrt{\Delta/(NT)}\sum_{i=1}^N\psi_i(u/\Delta)$ is tight in ${\mathcal D}[-C,C]$. Using Rosenthal's inequality (cf.\ Petrov (1995, p.\ 59)) we obtain that
\bals
E\left|\sum_{i=1}^N(\psi_i(t)-\psi_i(s))\right|^{\kappa/2}\leq c\left\{\sum_{i=1}^N E|\psi_i(t)-\psi_i(s)|^{\kappa/2}+\left(\sum_{i=1}^NE(\psi_i(t)-\psi_i(s))^2
\right)^{\kappa/4}
\right\}
\end{align*}
with some constant $c$. It is easy to see that
$$
|\psi_i(t)-\psi_i(s)|\leq \{|Q_i(t)(Q_i(t)-Q_i(s))|+|Q_i(s)(Q_i(t)-Q_i(s))|+|EQ_i^2(t)-EQ_i^2(s)|\}
$$
and
$$
|EQ_i^2(t)-EQ_i^2(s)|\leq c |t-s|.
$$
By the Cauchy--Schwarz inequality and Assumption \ref{cond-1}(ii) we have for all
$t,s \in[t_0-C/\Delta, t_0+C/\Delta]$
\bals
E(Q_i(t)(Q_i(t)-Q_i(s)))^2\leq (EQ_i^4(t)E(Q_i(t)-Q_i(s))^4)^{1/2}\leq c T|t-s|
\end{align*}
and therefore
$$
E(\psi_i(t)-\psi_i(s))^2\leq c T|t-s|
$$
where $c$ is a constant. Also, for $\kappa$ of Assumption \ref{cond-1}(ii)  we have
\bals
E|Q_i(t)(Q_i(t)-Q_i(s))|^{\kappa/2}\leq \{ E|Q_i(t)|^{\kappa}E|Q_i(t)-Q_i(s)|^{\kappa}\}^{1/2}\leq c\{ U_{i,\kappa}(t)U_{i,\kappa}(|t-s|)\}^{1/2}
\end{align*}
with some constant $c$. Thus we get via Assumption \ref{cond-1}(ii) that
$$
E\left\{\left|\left(\frac{\Delta}{NT}\right)^{1/2}\sum_{i=1}^N(\psi_i(u/\Delta)-\psi_i(v/\Delta))\right|^{\kappa/2}\right\}\leq c|u-v|^{\kappa/4},
$$
establishing tightness by Billingsley (1968, pp.\ 95 and 127). This also completes the proof of \eqref{w-3}.
\\
Following the arguments in the proof of \eqref{w-6} one can show that
\begin{align*}
\sup_{|t-t_0|\leq C/\Delta}|V^2(t)-V^2(t_0)|\leq 2\sup_{|t-t_0|\leq C/\Delta}|V(t)-V(t_0)|\sup_{|t-t_0|\leq C/\Delta}|V(t)|
=O_P(T^{1/2}/\Delta^{1/2}),
\end{align*}
and therefore \eqref{w-4} follows from Assumption \ref{cond-4}(ii).
To prove \eqref{w-5} we first write
\bals
Q_i(t)V(t)-Q_i(t_0)V(t_0)=V(t)(Q_i(t)-Q_i(t_0))+Q_i(t_0)(V(t)-V(t_0)).
\end{align*}
Repeating  the arguments used in the proof of \eqref{w-3} we obtain that
$$
\frac{1}{T}\sup_{|t-t_0|\leq C/\Delta}\left|V(t)\sum_{i=1}^N\gamma_i (Q_i(t)-Q_i(t_0))\right|=O_P(1)\left(\frac{\Gamma_{N,T}}{T\Delta}
\right)^{1/2}=o_P(1)
$$
via applying Assumption \ref{cond-4}(ii) and  $T\Delta\to \infty$.
\end{proof}

\medskip
Let
$$
R_{N,T}(u)=\sum_{i=1}^N\delta_i\calS_i(u/\Delta_{N,T}),\;\;u\geq 0,
$$
where $\calS_i(\cdot)$ is defined in \eqref{crlS}.

\medskip
\begin{lemma}\label{weak} If Assumptions \ref{b-1}, \ref{cond-1}, \ref{cond-tau}, \eqref{bar-tau} and \eqref{w-1} hold, then we have
$$
R_{N,T}(u)\;\;\stackrel{{\mathcal D}[0,C]}{\longrightarrow}\;\;\sigma W(u),
$$
for all $C>0$, where $W(u)$ stands for a Wiener process.
\end{lemma}

\begin{proof} For the sake of notational simplicity  we write $\Delta=\Delta_{N,T}$. Let $0=u_0<u_1<u_2<\ldots <u_k\leq C$ and
$\alpha_1, \alpha_2, \ldots , \alpha_k$. Under \eqref{w-1}  we write
$$
\sum_{\ell=1}^k \alpha_\ell(R_{ N,T}(u_\ell)- R_{ N,T}(u_{\ell-1}))=\sum_{i=1}^N\sum_{\ell=1}^k\alpha_\ell\delta_i(\calS_i(u_{\ell}/\Delta)-\calS_i(u_{\ell-1}/\Delta)).
$$
Using Assumptions \ref{b-1}, \ref{cond-1}, \ref{cond-tau}(i) and \eqref{w-1}, we get that
\begin{align*}
\sum_{i=1}^NE\left(\sum_{\ell=1}^k\alpha_\ell\delta_i(\calS_i(u_{\ell}/\Delta)-\calS_i(u_{\ell-1}/\Delta))\right)^2=\sigma^2 \sum_{\ell=1}^k \alpha_\ell^2(u_\ell-u_{{\ell-1}})(1+o(1)).
\end{align*}
Also, Assumptions \ref{b-1}(ii) and \ref{cond-tau}(ii) imply
\begin{align}\label{rose-1}
\sum_{i=1}^N&E\left|\delta_i\sum_{\ell=1}^k\alpha_\ell(\calS_i(u_{\ell}/\Delta)-\calS_i(u_{\ell-1}/\Delta))\right|^{\bar{\tau}}\\
&\leq c\sum_{i=1}^N|\delta_i|^{\bar{\tau}}\max_{1\leq \ell\leq k}U_{i,\bar{\tau}}(|u_\ell-u_{\ell-1}|/\Delta)
\notag\\
&\leq c\left\{
\Delta^{-\bar{\tau}/2}\sum_{i=1}^N|\delta_i|^{\bar{\tau}} \right\}\notag
\end{align}
So using Lyapunov's theorem (cf.\ Petrov (1995,\ p.\ 154)) we conclude  via \eqref{bar-tau} that
$$
\sum_{\ell=1}^k \alpha_\ell(R_{ N,T}(u_\ell)- R_{ N,T}(u_{\ell-1}))\;\;\stackrel{{\mathcal D}}{\to}\;\;\;
\sigma\sum_{\ell=1}^k \alpha_\ell(W(u_\ell)- W(u_{\ell-1})),
$$
where $W$ stands for a Wiener process. Applying the Cram\'er--Wold theorem (cf.\ Billingsley (1968, p.\ 49)) we obtain that the finite dimensional distributions of $R_{N,T}(u)$ converge to that of $\sigma W(u)$. Next we show that $R_{N,T}(u)$ is tight in ${\mathcal D}[0,C]$. Following the arguments in \eqref{rose-1}, Rosenthal's inequality yields for all $0\leq u,  v \leq C$
\begin{align*}
E|R_{N,T}(u)&-R_{N,T}(v)|^{\bar{\tau}}\\
&\leq c \Delta^{-\bar{\tau}/2}\left\{\sum_{i=1}^N|\delta_i|^{\bar{\tau}}E|\calS_i(u/\Delta)-\calS_i(v/\Delta)|^{\bar{\tau}}
+\left(\sum_{i=1}^N\delta_i^2E(\calS_i(u/\Delta)-\calS_i(v/\Delta))^2\right)^{\bar{\tau}/2}
\right\}\\
&\leq c \left\{\sum_{i=1}^N|\delta_i|^{\bar{\tau}}U_{i,\bar{\tau}}(|u-v|/\Delta)
+\left(\sum_{i=1}^N\delta_i^2U_{i,2}(|u-v|/\Delta)\right)^{\bar{\tau}/2}
\right\}\\
&\leq c\left\{\Delta^{-\bar{\tau}/2} \sum_{i=1}^N|\delta_i|^{\bar{\tau}}+1\right\}|u-v|^{\bar{\tau}/2}\\
&\leq c |u-v|^{\bar{\tau}/2}
\end{align*}
on account of Assumption \ref{cond-tau}(ii) and \eqref{bar-tau}. The tightness now follows from Billingsley (1968,\ p.\ 127).
\end{proof}

\begin{lemma}\label{weak-1} If Assumptions \ref{b-1}, \ref{cond-1}, \ref{cond-tau} and \eqref{bar-tau}  hold, then for all integers $0<t_1<t_2<\ldots <t_K$ we have that
$$
\left( \sum_{i=1}^N\delta_i\calS_i(t_\ell), 1\leq \ell\leq K\right) \;\;\stackrel{{\mathcal D}}{\longrightarrow}\;\;\left( {\frak G}(t_\ell), 1\leq \ell \leq K \right),
$$
where the Gaussian process ${\frak G}(t), t=0, \pm 1, \pm 2. \ldots$ is defined in Theorem \ref{w-fac}.
\end{lemma}
\begin{proof} We repeat the first half of the proof of Lemma \ref{weak}. The result of Lemma \ref{weak-1} follows from \eqref{rose-1} and Lyapunov's central limit theorem due to assumption \eqref{bar-tau}.
\end{proof}

\noindent
{\it Proof of Theorem \ref{w-fac}.} Let $\Delta=\Delta_{N,T}$. It follows from Assumption \ref{b-1}(ii) and Lemma \ref{weak} that for all $C>0$
\beq\label{cal-conv}
 \sum_{i=1}^N\delta_i(\calS(t_0+u/\Delta)-\calS(t_0))\;\;\;\stackrel{{\mathcal D}[-C,C]}{\longrightarrow}\;\;\;\;\sigma W(u),
\eeq
where $W(u), -\infty<u<\infty$ is a two sided Wiener process. Also, $\Delta g_\theta(u/\Delta)=g_\theta(u)$ and since ${\frak s}=0$ by \eqref{V-comp} we have that
 \beq\label{smalls}
\Sigma_{N,T} \sup_{|t-t_0|\leq \Delta}|V(t)-V(t_0)|=o_P(1).
\eeq
By Lemma \ref{lem-w} we conclude that  for all $C>0$
\beq\label{cont-1}
\frac{1}{T}\left( U_N(t_0+u/\Delta) -U_N(t_0) \right)\;\;\;\stackrel{{\mathcal D}[-C,C]}{\longrightarrow}\;\;\;\;2\theta(1-\theta)(\sigma W(u)-g_\theta(u)).
\eeq
By the continuous mapping theorem we conclude from \eqref{cont-1} that for all $C$
\beq\label{cont-2}
\mbox{argmax}_{|t-t_0|\leq C/\Delta}\left( U_N(t_0+u/\Delta) -U_N(t_0) \right)\;\;\;\stackrel{{\mathcal D}}{\longrightarrow}\;\;\;\;\mbox{argmax}_{|u|\leq C}(\sigma W(u)-g_\theta(u)).
\eeq
According to the law of iterated logarithm, we have that
\beq\label{cont-3}
\lim_{C\to \infty} \mbox{argmax}_{|u|\leq C}(\sigma W(u)-g_\theta(u))\;\;\;\to\;\;\;\mbox{argmax}_{u}(\sigma W(u)-g_\theta(u))\;\;\;\mbox{a.s.}
\eeq
Now \eqref{f-1} follows from Lemma \ref{rate-con}, \eqref{cont-2} and \eqref{cont-3}.\\
It follows from  Assumption \ref{b-1}(ii), \eqref{smalls} and Lemmas \ref{lem-w}, \ref{weak-1} that for every integer  $C>0$
\begin{align}\label{dist-1}
&\left\{\frac{1}{T}\left( U_N(t_0+t) -U_N(t_0) \right), \;t=0, \pm 1, \pm 2, \ldots ,\pm C\right\}\\
&\;\;\;\;\;\;\;\stackrel{{\mathcal D}}{\longrightarrow}\;\;\;\;\left\{2\theta(1-\theta)({\frak G}(t)-{\frak d}g_\theta(t)), \; t=0, \pm 1, \pm 2, \ldots ,\pm C\right\}.\notag
\end{align}
Observing that $\calu(t,t)=O(t)$, the normality of ${\frak G}(t)$ with the Borel--Cantelli lemma yields that
$
\lim_{|t|\to \infty}{\frak G}(t)/t =0,
$
and therefore
$$
\lim_{C\to \infty}\mbox{argmax}_{|t|\leq C}({\frak G}(t)-{\frak d}g_\theta(t))=\mbox{argmax}_{t}({\frak G}(t)-{\frak d}g_\theta(t))\;\;\;\;\mbox{a.s.}
$$
The proof of \eqref{f-2} is now completed via Lemma \ref{rate-con}.
\qed
\\

\medskip
\noindent
{\it Proof of Theorem \ref{w-fac-mid}.} It follows from Assumption \ref{total} that the  Wiener processes in Assumption \ref{weak-v} and \eqref{cal-conv} are independent. Hence Lemma \ref{lem-w} yields for all $C>0$ that
$$
\frac{1}{T}\left( U_N(t_0+u/\Delta) -U_N(t_0) \right)\;\;\;\stackrel{{\mathcal D}[-C,C]}{\longrightarrow}\;\;\;\;2\theta(1-\theta)(\sigma^2+{\frak s}^2)^{1/2} W(u)-g_\theta(u)),
$$
where $W(u), -\infty<u<\infty$ is a two--sided Wiener process. Arguments used in \eqref{cont-2} and \eqref{cont-3} could be repeated to finish the proof of \eqref{f-1-m}.\\
Referring  again to Assumption \ref{total} it is immediate that the Gaussian process ${\frak G}(t)$ and ${\mathcal V}(t)$ are independent. So applying Lemma \ref{lem-w} we replace \eqref{dist-1} with
\begin{align*}
&\left\{\frac{1}{T}\left( U_N(t_0+t) -U_N(t_0) \right), \;t=0, \pm 1, \pm 2, \ldots ,\pm C\right\}\\
&\;\;\;\;\;\;\;\stackrel{{\mathcal D}}{\longrightarrow}\;\;\;\;\left\{2\theta(1-\theta)({\frak G}(t)+{\frak s}{\frak d}^{1/2}{\mathcal V}(t)-{\frak d}g_\theta(t)), \; t=0, \pm 1, \pm 2, \ldots ,\pm C\right\}\notag
\end{align*}
for any $C>0$. Observing that ${\mathcal V}(t)/t \to 0$ a.s.\ we need only minor modifications of the proof of \eqref{f-2} to complete the proof of \eqref{f-2-m}.
\qed
\medskip

\begin{lemma}\label{rate-con-large} We assume that Assumptions \ref{b-1}--\ref{m-cho-1} hold, and $|{\frak s}|=\infty.$
Then, as $N,T\to \infty$ we have that
\beq\label{rateconsig}
 |\hat{t}_{N,T}-t_0|=O_P(M_{N,T}),
\eeq
where
$M_{N,T}= (\Sigma_{N,T}/\Delta_{N,T})^2.$
\end{lemma}

\begin{proof}  By Lemma \ref{first} it is enough to prove that for all $0<\alpha<\theta$
\beq\label{rateconsig-1}
 |\tilde{t}_{N,T}(\alpha)-t_0|=O_P(M_{N,T}).
\eeq
The result follows from Lemmas \ref{tt-0}--\ref{ga-de} with $M=(\Sigma_{N,T}/\Delta_{N,T})^2.$
\end{proof}
\medskip
\noindent
{\it Proof of Theorem \ref{w-fac-strong}.} Let $M=(\Sigma_{N,T}/\Delta_{N,T})^2$. Since $\Delta_{N,T}$ is bounded, by \eqref{sidestrong} we have
\beq\label{bounded}
\frac{\Sigma_{N,T}}{\Delta_{N,T}}\to \infty.
\eeq
Following the proof of Lemma \ref{lem-w} one can show that for all $C>0$
\begin{align}\label{w-2-**}
\sup_{|t-t_0|\leq CM}\biggl|\frac{1}{T}\sum_{i=1}^N\delta_i(Q_i(t)r(t)-Q_i(t_0)r(t_0))
\biggl|
=o_P(1),
\end{align}
\beq\label{w-3-*}
\sup_{|t-t_0|\leq CM}\left|\frac{1}{T}\sum_{i=1}^N(Q_i^2(t)-Q_i^2(t_0))\right|=o_P(1),
\eeq
\beq\label{w-4-*}
\sup_{|t-t_0|\leq CM}\left|\frac{1}{T}\sum_{i=1}^N\gamma_i^2(V^2(t)-V^2(t_0))\right|=o_P(1)
\eeq
and
\beq\label{w-5-*}
\sup_{|t-t_0|\leq CM}\left|\frac{1}{T}\sum_{i=1}^N \gamma_i(Q_i(t)V(t)-Q_i(t_0)V(t_0))\right|=o_P(1).
\eeq
It follows from Assumptions \ref{con-eta} and \ref{weak-v}  that for all $C>0$
\begin{align*}
&\frac{1}{T}\left(\Delta_{N,T}(r^2(t_0+uM)-r^2(t_0))+2\Sigma_{N,T}(V(t_0+uM)r(t_0+uM)-V(t_0)r(t_0))\right)\\
&\quad \quad\;\;\;\stackrel{{\mathcal D}[-C,C]}{\longrightarrow}\;\;\;\;2\theta(1-\theta)( W(u)-g_\theta(u)),
\end{align*}
where $W(u), -\infty <u<\infty$ denotes a two--sided Wiener process. Using now \eqref{w-2-**}--\eqref{w-5-*}
$$
\frac{1}{T}\left( U_N(t_0+uM) -U_N(t_0) \right)\;\;\;\stackrel{{\mathcal D}[-C,C]}{\longrightarrow}\;\;\;\;2\theta(1-\theta)( W(u)-g_\theta(u)).
$$
Arguments used in \eqref{cont-2} and \eqref{cont-3} could be repeated to complete the proof  of Theorem \ref{w-fac-strong}.
\qed

\section{Proof of Theorem \ref{th-cons}}\label{sec-cons} For the sake of brevity we use $\hat{t}$, $\Delta$ and $\Sigma$  for $\hat{t}_{N,T}$,  $\Delta_{N,T}$ and $\Sigma_{N,T}$, respectively  . We start with the proof of \eqref{d-cons}.
Let $M=M(N,T)$ be a sequence satisfying
\beq\label{m-def} M\to\infty, \quad
M/T\to 0,
\eeq
if the conditions of Theorem \ref{w-fac} or \ref{w-fac-mid} hold and
\beq\label{m-def-dep} M\to\infty, \quad
M/\min(T, \Sigma^2/\Delta^2)\to 0
\eeq
under the assumptions of Theorem \ref{w-fac-strong}. First we show that for  all $M$ satisfying   $M/T\to 0$  we have that
\beq\label{li-1}
\sup_{|u|\leq M}\left|\frac{1}{\Delta}\sum_{i=1}^N\left(\frac{1}{t_0+u}\sum_{1\leq t \leq t_0+u}X_{i,t}-\frac{1}{T-(t_0+u)}\sum_{t_0+u<t\leq T}X_{i,t}\right)^2-1\right|=o_P(1).
\eeq
Using \eqref{model} we have for all $0<u\leq M$
\begin{align*}
\frac{1}{t_0+u}&\sum_{1\leq t \leq t_0+u}X_{i,t}-\frac{1}{T-(t_0+u)}\sum_{t_0+u<t\leq T}X_{i,t}=\left(\frac{u}{t_0+u}-1\right)\delta_i+\frac{\gamma_i}{t_0+u}\sum_{1\leq t \leq t_{0}+u}\eta_t\\
&+\frac{1}{t_0+u}\sum_{1\leq t \leq t_{0}+u}e_{i,t}-\frac{\gamma_i}{T-(t_0+u)}\sum_{T-(t_0+u)<t\leq T}\eta_t-\frac{1}{T-(t_0+u)}\sum_{T-(t_0+u)<t\leq T}e_{i,t}.
\end{align*}
Since $M/T\to 0$, we have
\beq\label{li-2}
\sup_{0<u\leq M}\left|\left(\frac{t_0}{t_0+u}\right)^2-1\right|\to 0.
\eeq
Applying Assumption \ref{con-eta} with Markov's and the maximal inequality of M\'oritz et al (1982)
we obtained for all $z>0$
\begin{align}\label{li-3}
P\left\{\frac{\Gamma}{T^2\Delta}\sup_{0<u\leq M}\left(\sum_{1\leq t \leq t_{0}+u}\eta_t\right)^2\geq z
\right\}&=P\left\{\sup_{0<u\leq M}\left(\sum_{1\leq t \leq t_{0}+u}\eta_t\right)^{\bar{\kappa}}\geq (zT^2\Delta/\Gamma)^{\bar{\kappa}/2}
\right\}\\
&\leq \left(\frac{\Gamma}{T^2\Delta}\right)^{\bar{\kappa}/2}E\sup_{0<u\leq M}\left(\sum_{1\leq t \leq t_{0}+u}\eta_t\right)^{\bar{\kappa}}\notag\\
&=O(1)\left(\frac{\Gamma}{T\Delta}\right)^{\bar{\kappa/2}}\to 0\notag
\end{align}
on account of Assumption \ref{cond-4}. Following the proof of \eqref{li-3} but now using Assumptions \ref{b-1}(i) and \ref{cond-1}, we conclude
\begin{align}\label{li-4}
P\left\{\frac{1}{T^2\Delta}\sum_{i=1}^N\left(\sum_{1\leq t \leq t_{0}}e_{i,t}\right)^2\geq z
\right\}
&\leq \frac{1}{zT^2\Delta}\sum_{i=1}^NE\left(\sum_{1\leq t \leq t_{0}}e_{i,t}\right)^2 \\
&=O(1)\frac{N}{zT\Delta}\to 0\notag
\end{align}
by Assumption \ref{cond-4}(i). The stationarity in Assumption \ref{b-1}(ii) yields
\begin{align}\label{li-5}
P&\left\{\frac{1}{T^2\Delta}\sup_{0<u\leq M}\sum_{i=1}^N\left(\sum_{t_0+1\leq t \leq t_{0}+u}e_{i,t}\right)^2\geq z
\right\}\\
&=P\left\{\frac{1}{T^2\Delta}\sup_{0<u\leq M}\sum_{i=1}^N\left(\sum_{1\leq t \leq u}e_{i,t}\right)^2\geq z
\right\}\notag\\
&\leq \sum_{u=1}^MP\left\{\frac{1}{T^2\Delta}\sum_{i=1}^N\left(\sum_{1\leq t \leq u}e_{i,t}\right)^2\geq z
\right\}\notag\\
&=O(1)\frac{NM}{zT^2\Delta}\to 0\notag
\end{align}
by Assumption \ref{cond-4}(i) and the assumption that $M/T\to 0$. Putting together \eqref{li-4} and \eqref{li-5} we obtain that
\beq\label{li-6}
\sup_{0<u \leq M}\frac{1}{\Delta}\sum_{i=1}^N\left(\frac{1}{t_0+u}\sum_{1\leq t \leq t_0+u}e_{i,t}\right)^2=o_P(1).
\eeq
Following the proofs of  \eqref{li-3} and \eqref{li-6} one can prove that
\beq\label{li-7}
\sup_{0<u \leq M}\frac{1}{\Delta}\sum_{i=1}^N\left( \frac{\gamma_i}{T-(t_0+u)}\sum_{T-(t_0+u)<t\leq T}\eta_t  \right)^2=o_P(1)
\eeq
and
\beq\label{li-8}
\sup_{0<u \leq M}\frac{1}{\Delta}\sum_{i=1}^N\left( \frac{1}{T-(t_0+u)}\sum_{T-(t_0+u)<t\leq T}e_{i,t} \right)^2=o_P(1).
\eeq
The result in \eqref{li-1} now follows from \eqref{li-2}, \eqref{li-3} and \eqref{li-6}--\eqref{li-8}.
Since under our conditions
\beq\label{tt-con}
|\hat{t}-t_0|/M=o_P(1),
\eeq
 the proof of \eqref{d-cons} is complete.\\

By \eqref{summa} we have
\begin{align*}
U_N(\hat{t}+v)-U_N(\hat{t})=U_N^{(1)}(v)+\ldots +U_N^{(6)}(v),
\end{align*}
where
\begin{align*}
&U_N^{(1)}(v)=\sum_{i=1}^N\delta_i^2(r^2(\hat{t}+v)-r^2(\hat{t})), \quad U_N^{(2)}(v)=\sum_{i=1}^N(Q_i^2(\hat{t}+v)-Q_i^2(\hat{t})),\\
&U_N^{(3)}(v)=\sum_{i=1}^N\gamma_i^2(V^2(\hht+v)-V^2(\hht)), \quad U_N^{(4)}(v)=2\sum_{i=1}^N\gamma_i(Q_i(\hht+v)V(\hht+v)-Q_i(\hht)V(\hht)),\\
&U_N^{(5)}(v)=2\sum_{i=1}^N\delta_i(r(\hht+v)Q_i(\hht+v)-r(\hht)Q_i(\hht)),
\end{align*}
and
$$
U_N^{(6)}(v)=2\sum_{i=1}^N\delta_i\gamma_i(r(\hht+v)V(\hht+v)-r(\hht)V(\hht)).
$$
It follows from the proofs of Lemmas  \ref{qq-0}, \ref{gV}, \ref{ga-q} and \ref{lem-w} that for all $M$ satisfying \eqref{m-def} we have
$$
\sup_{|v|\leq M}\frac{1}{|v|\hat{r}^2_{N,T}}\left\{|U_N^{(2)}(v)| + |U_N^{(3)}(v)|+ |U_N^{(4)}(v)| \right\}^2=o_P(\Xi_{N,T}),
$$
where $\hat{r}_{N,T}$ is defined in \eqref{hat-r}. Applying now Lemmas \ref{r-q}, \ref{ga-de} and \ref{lem-w} we conclude
$$
\sup_{|v|\leq M}\frac{1}{|v|\hat{r}^2_{N,T}}\left| U_N^{(5)}(v)- U_N^{(7)}(v)\right|^2=o_P(\Xi_{N,T} )
$$
and
$$
\sup_{|v|\leq M}\frac{1}{|v|\hat{r}_{N,T}}\left| U_N^{(6)}(v)- U_N^{(8)}(v)\right|^2=o_P(\Xi_{N,T} ),
$$
where
$$
U_N^{(7)}(v)=2r(t_0)\sum_{i=1}^N\delta_i{\mathcal A}_i(v;\hht)\quad \mbox{and} \quad   U_N^{(8)}(v)=2r(t_0)\sum_{i=1}^N\delta_i\gamma_i{\mathcal B}(v;\hht),
$$
where
\begin{displaymath}
{\mathcal A}_i(v;t)=\left\{
\begin{array}{ll}
\displaystyle \sum_{s=t +1}^{t +v}e_{i,s},\quad &v>0
\vspace{.3cm}\\
0, &v=0
\vspace{.3cm}\\
\displaystyle\sum^{t -1}_{s=t +v}e_{i,s}, &v<0,
\end{array}
\right.
\end{displaymath}
and
\begin{displaymath}
{\mathcal B}(v;t)=\left\{
\begin{array}{ll}
\displaystyle \sum_{s=t +1}^{t +v}\eta_{s},\quad &v>0
\vspace{.3cm}\\
0, &v=0
\vspace{.3cm}\\
\displaystyle\sum^{t -1}_{s=t +v}\eta_{s}, &v<0,
\end{array}
\right.
\end{displaymath}
 Using \eqref{m-def} and \eqref{tt-con} one can verify that
$$
\sup_{ |v|\leq M}\frac{1}{|v|\hat{r}_{N,T}}\left| U_N^{(7)}(v)- U_N^{(9)}(v)\right|^2=o_P(\Xi_{N,T})
$$
and
$$
\sup_{ |v|\leq M}\frac{1}{|v|\hat{r}_{N,T}}\left| U_N^{(8)}(v)- U_N^{(10)}(v)\right|^2=o_P(\Xi_{N,T}),
$$
where
$$
U_N^{(9)}(v)=2r(t_0)\sum_{i=1}^N\delta_i{\mathcal A}_i(v;t_0)\quad \mbox{and} \quad   U_N^{(10)}(v)=2r(t_0)\sum_{i=1}^N\delta_i\gamma_i{\mathcal B}(v;t_0).
$$
Let $m=m(N,T)\leq M$ and $m\to \infty$, as $\min(N,T)\to \infty$. Using Assumptions \ref{cond-1} and \ref{con-eta}  we get
$$
 \sup_{m\leq |v|\leq M}\left|\frac{1}{4|v|r^2(t_0)}E\left(U_N^{(9)}(v)+U_N^{(10)}(v)\right)^2-\Xi_{N,T}\right|=o(\Xi_{N,T}).
$$
and
\beq\label{uni-v-0}
\sup_{|v|\leq m}\frac{1}{4|v|r^2(t_0)\Xi_{N,T}}\left(U_N^{(9)}(v)+U_N^{(10)}(v)\right)^2=O_P(1).
\eeq

We claim that
\beq\label{uni-v}
\sup_{m\leq |v|\leq M}\left|\frac{1}{4|v|r^2(t_0)\Xi_{N,T}}\left(U_N^{(9)}(v)+U_N^{(10)}(v)\right)^2-1\right|=o_P(1).
\eeq
The statement in \eqref{uni-v} is a uniform weak law of large numbers, so we can repeat the proof  of \eqref{li-2}  to prove it. Namely, due to stationarity, it follows from Assumptions \ref{cond-1} and \ref{con-eta} that for every $v\in [-M, \ldots , -m, m, \ldots ,M]$ that
\beq\label{uni-v-1}
\frac{1}{4|v|r^2(t_0)\Xi_{N,T}}\left(U_N^{(9)}(v)+U_N^{(10)}(v)\right)^2\;\;\stackrel{P}{\to}\;\;1.
\eeq
Now \eqref{uni-v} follows from \eqref{uni-v-1} if $ \left(U_N^{(9)}(v)+U_N^{(10)}(v)\right)^2/(|v|r^2(t_0)\Xi_{N,T}), m\leq |v|\leq M$ is tight. The tightness can be proven along the lines of  the proofs of Lemmas \ref{r-q} and \ref{ga-de}. Our arguments show that
\begin{align}\label{ff-1}
\sup_{ |v|\leq M}\Biggl| \frac{1}{4|v|\hat{r}^2_{N,T}}&\left(U_N(\hat{t}+v)-U_N(\hht)-\Delta_{N,T}(r^2(\hht+v)-r^2(\hht))\right)^2\\
&-\frac{1}{|v|r(t_0)}\left(U_N^{(9)}(v)+U_N^{(10)}(v)\right)^2\Biggl|=o_P(\Xi_{N,T}).\notag
\end{align}
Putting together  \eqref{uni-v-0}, \eqref{uni-v} and \eqref{ff-1}, the result in \eqref{xi-const} follows.
\medskip

\end{document}